\documentclass[a4paper,11pt]{article}
\usepackage{amsmath,amsthm, authblk}
\usepackage{amssymb,latexsym}
\usepackage{graphicx,subfig}
\usepackage{epsfig}
\usepackage{color}

\newtheorem{thm}{Theorem}[section]

\newtheorem{lma}[thm]{Lemma}
\newtheorem{cor}[thm]{Corollary}
\newtheorem{prp}[thm]{Proposition}
\newtheorem{exm}[thm]{Example}
\newtheorem{clm}[thm]{Claim}

\newtheorem{conj}[thm]{Conjecture}

\def\eps{{\varepsilon}}

\numberwithin{equation}{section}

\title{A Dirac type condition for properly coloured paths and cycles}
\author{Allan Lo}
\affil{School of Mathematics, University of Birmingham, \\Birmingham, B15 2TT, UK\\
\texttt{s.a.lo@bham.ac.uk}}

\begin{document}

\maketitle

\abstract{
Let $c$ be an edge-colouring of a graph~$G$ such that for every vertex~$v$ there are at least $d \ge 2$ different colours on edges incident to~$v$.
We prove that $G$ contains a properly coloured path of length~$2d$ or a properly coloured cycle of length at least~$d+1$.
Moreover, if $G$ does not contain any properly coloured cycle, then there exists a properly coloured path of length $3 \times 2^{d-1}-2$.
}

\section{Introduction} \label{sec:introduction}
All graphs considered in this paper are simple without loops unless stated otherwise.
Throughout this paper, $G$ is assumed to be a graph.
An \emph{edge-colouring}~$c$ of $G$ is an assignment of colours to the edges of~$G$.
An \emph{edge-coloured graph} is a graph $G$ with an edge-colouring~$c$ of $G$.

An edge-coloured graph $G$ is said to be \emph{properly coloured}, or \emph{p.c.} for short, if no two adjacent edges have the same colour.
Moreover, $G$ is \emph{rainbow} if every edge has distinct colour.
The \textit{colour degree}~$d^c_G(v)$ of a vertex~$v$ is the number of different colours on edges incident to~$v$.
The \textit{minimum colour degree} $\delta^c(G)$ of a graph~$G$ is the minimal $d^c_G(v)$ over all vertices~$v$ in~$G$.
In this article, we study the p.c. paths and p.c. cycles in edge-coloured graphs~$G$ with $\delta^c(G) \ge 2$.
For surveys regarding properly coloured subgraphs and rainbow subgraphs in edge-coloured graphs, we recommend Chapter~16 of~\cite{MR2472389} and \cite{MR2438857} respectively.

Grossman and H{\"a}ggkvist~\cite{MR701173} gave a sufficient condition for the existence of a p.c. cycle in edge-coloured graphs with two colours.
Later on, Yeo~\cite{MR1438622} extended the result to edge-coloured graphs with any number of colours.

\begin{thm}[Grossman and H{\"a}ggkvist~\cite{MR701173}, Yeo~\cite{MR1438622}] \label{thm:Yeo}
Let $G$ be a graph with an edge-colouring~$c$. 
If $G$ has no properly coloured cycle, then there is a vertex~$z$ in~$G$ such that no connected component of $G-z$ is joined to $z$ with edges of more than one colour. 
\end{thm}

Bollob\'as and Erd\H{o}s~\cite{MR0411999} proved that if $n\ge 3$ and $\delta^c(K_n) \ge 7n/8$, then there exists a p.c. Hamiltonian cycle.
(A path or cycle is Hamiltonian if it spans all the vertices.)
They also asked the question of whether $\delta^c(K_n)\ge \lceil (n+5)/3 \rceil$ guarantees a p.c. Hamiltonian cycle.
Fujita and Magnant~\cite{FujitaMagnant09} showed that $\delta^c(K_n)= \lfloor n/2 \rfloor$ is not sufficient by constructing an edge-colouring~$c$ of~$K_{2m}$ with $\delta^c(K_{2m}) =m$, which has no p.c. Hamiltonian cycle.
Alon and Gutin~\cite{MR1610269} proved that for every $\eps >0$ and $n > n_0(\eps)$ if no vertex in an edge-coloured $K_n$ is incident with more than $(1-1/\sqrt{2} - \eps)n$ edges of the same colour, then there exists a p.c. Hamiltonian cycle.
This easily implies that if $\delta^c(K_n)\ge (1/\sqrt{2} + \eps)n$ then there is a p.c. Hamiltonian cycle.

Li and Wang~\cite{MR2519172} proved that if $\delta^c(G) \ge d \ge 2$, then $G$ contains a p.c. path of length~$2d$ or a p.c. cycle of length at least~$\lceil 2d/3 \rceil +1$.
We strengthen the bound of Li and Wang~\cite{MR2519172} to the best possible value.
Our proof begins with the rotation-extension technique of P\'{o}sa~\cite{MR0389666}, which we adapt for use on edge-coloured graphs.

\begin{thm} \label{thm:2d+1}
Every edge-coloured graph $G$ with $\delta^c(G) \ge 2$ contains a properly coloured path of length~$2\delta^c(G)$ or a properly coloured cycle of length at least~$\delta^c(G)+1$.
\end{thm}

Note that a disjoint union of rainbow $K_{d+1}$ has minimum colour degree~$d$.
The longest p.c. path and p.c. cycle have lengths $d$ and $d+1$ respectively.
Together with the following example, we conclude that Theorem~\ref{thm:2d+1} is best possible.

\begin{exm} \label{exm:1}
For integers $p \ge d \ge 2$, define the edge-coloured graph $\widetilde{G}(d;p)$ as follows: take a new vertex~$x$ and $p$ vertex-disjoint rainbow copies of $K_{d}$, $H_1, H_2, \dots, H_p$, add an edge of new colour $c_j$ between $x$ and every vertex of $H_j$ for each~$j$.
It is easy to see that $\delta^c(\widetilde{G}(d;p)) = d$.
Note that $\widetilde{G}(d;p)$ consists of $p$ copies of $K_{d+1}$ intersecting at one vertex, namely $x$.
Hence, $\widetilde{G} (d,p)$ has no path of length $2d+1$.
Every cycle of length at least $d+1$ in $G$ contains $x$ and is not properly coloured.
Therefore, $\widetilde{G}(d;p)$ does not contain p.c. cycles of length at least~$d+1$.
\end{exm}

In a non-edge-coloured graph~$G$, it is a trivial fact that if $\delta(G) \ge 2$, then $G$ contains a cycle.
However, there exist edge-coloured graphs $G$ with $\delta^c(G) \ge 2$ that do not contain any p.c. cycles, e.g. $\widetilde{G}(2;p)$.
Given an integer~$k\ge 3$ and an edge-colouring~$c$ of a graph~$G$ such that no p.c. cycle in~$G$ has length at least~$k$, it is natural to ask for the length of the longest p.c. path in $G$.
We prove that the longest p.c. path grows exponentially with~$\delta^c(G)$ for fixed $k$.

\begin{thm} \label{thm:k<d}
For integers $k \ge 3$, every edge-coloured graph $G$ with $\delta^c(G) \ge \lceil 3k/2 \rceil-3$ contains a properly coloured path of length~$k2^{\delta^c(G)-\lceil 3k/2 \rceil+4}-2$ or a properly coloured cycle of length at least~$k$.
\end{thm}

On the other hand, we show that there exist edge-coloured graphs~$G$, which only contain p.c. paths and p.c. cycles of lengths at most $k 2^{\delta^c(G)-k+2}-2$ and~$k-1$ respectively.

\begin{prp} \label{prp:upper}
For integers $d \ge k-1 \ge 2$, there exist infinitely many edge-coloured graphs $G$ with $\delta^{c}(G) \ge d$ such that each properly coloured path has length at most $k 2^{d-k+2}-2$ and no properly coloured cycle has length longer than $k-1$.
\end{prp}

For $k=3$, Theorem~\ref{thm:k<d} gives the following simple corollary, which is best possible by Proposition~\ref{prp:upper}.

\begin{cor} \label{cor:k=3}
Every edge-coloured graph $G$ contains a properly coloured path of length~$3 \times 2^{\delta^c(G)-1}-2$ or a properly coloured cycle.
\end{cor}

For $d+1 \ge k \ge 3$, we conjecture the following result.

\begin{conj} \label{conj:k<d}
For integers $k \ge 3$, every edge-coloured graph $G$ with $\delta^c(G) \ge k-1$ contains a properly coloured path of length~$k2^{\delta^c(G)-k+2}-2$ or a properly coloured cycle of length at least~$k$.
\end{conj}

This conjecture is true for $k=3$, $k=4$ and $k=\delta^c(G)+1$ by Corollary~\ref{cor:k=3}, Theorem~\ref{thm:k<d} and Theorem~\ref{thm:2d+1} respectively.
Moreover, if Conjecture~\ref{conj:k<d} is true, then it is best possible by Proposition~\ref{prp:upper}.

We are also interested in the longest p.c. path in~$G$ with $\delta^c(G) = d$ without any constraint on p.c. cycles.
Trivially, if $G$ is a disjoint union of rainbow~$K_{d+1}$, then the longest p.c. path has length $d$.
We add the assumption that $G$ is connected to avoid the trivial answer just given.
The following example shows that there are connected graphs whose p.c. paths have length at most $\lfloor 3d/2 \rfloor$.

\begin{exm} \label{exm:longpath}
For integers $d \ge 1$ and $n \ge 3d/2$, define the edge-coloured graph $\widehat{G} = \widehat{G}(d;n)$ on $n$ vertices as follows.
Partition vertex set of $G$ into $X$ and $Y$ with $X = \{x_1, x_2,\dots, x_d\}$.
The subgraph induced by vertex set $X$ is a rainbow $K_d$.
The subgraph induced by vertex set $Y$ is empty.
For each $1 \le i \le d$, add an edge of new colour~$c_i$ between $x_i$ and each $y \in Y$.
By our construction, $\delta^c(\widehat{G}) = d$.
Note that every p.c. path in $\widehat{G}$ with both endpoints in $Y$ must contain at least two vertices in~$X$.
Thus, every p.c. path in $\widehat{G}$ is of length at most $|X| + \lfloor |X|/2 \rfloor = \lfloor 3d/2 \rfloor$.
\end{exm}

We believe that the example above is best possible and conjecture the following.

\begin{conj} \label{conj:path}
Every edge-coloured connected graph $G$ contains a properly coloured Hamiltonian cycle or a properly coloured path of length $\lfloor 3\delta^c(G)/2 \rfloor$.
\end{conj}

This conjecture can be easily verified for $\delta^c(G) \le 3$.
By case analysis, we can show that $G$ contains a path of length $\delta^c(G)+2$ if $|G| \ge \delta^c(G) + 3$ and $\delta^c(G) \ge 4$.
Therefore, the conjecture is true for $\delta^c(G) \le 5$.
However, for $\delta^c(G) \ge 6$ we are only able to show that $G$ contains a p.c. path of length at least $ 6\delta^c(G)/5 -1$ or a p.c. Hamiltonian cycle.

\begin{thm} \label{thm:path}
Every edge-coloured connected graph $G$ contains a properly coloured Hamiltonian cycle or a properly coloured path of length at least $6\delta^c(G)/5-1$.
\end{thm}

We set up notation and tools in the next section.
We prove Theorem~\ref{thm:2d+1} in Section~\ref{sec:2d+1}.
In Section~\ref{sec:k<d}, we prove Theorem~\ref{thm:k<d} and Proposition~\ref{prp:upper}.
Theorem~\ref{thm:path} is proved in Section~\ref{sec:path}.
Finally, we consider a variant of colour degree in Section~\ref{sec:deg} and give a counterexample to a conjecture in~\cite{Manoussakis10}.


\section{Preliminaries} \label{sec:prelimiary}
For $a,b \in \mathbb{N}$, let $[a,b]$ and $[b]$ denote the sets $\{ i \in \mathbb{N} : a \le i \le b\}$ and $\{ i \in \mathbb{N} : 1 \le i \le b\}$ respectively.

Given a graph~$G$, $V(G)$ and $E(G)$ denote the sets of vertices and edges of~$G$ respectively.
Denote the order of $G$ by $|G|$. 
Given a vertex subset $U \subseteq V(G)$, $G[U]$ is the (edge-coloured) subgraph of $G$ induced by~$U$.
Given an edge-colouring $c$ of~$G$, a \textit{colour neighbourhood}~$N^c_G(v)$ of a vertex~$v$ is a maximal subset of the neighbourhood of~$v$ such that $c(v,w_1) \ne c(v,w_2)$ for all distinct $w_1, w_2 \in N^c_G(v)$.
Thus, $|N^c_G(v)|=d^c_G(v)$.
It should be noted that there is a choice on $N^c_G(v)$, which we will specify later.
If the edge-coloured graph $G$ is clear from the context, then we omit the subscript.

For convenience, let the vertices of~$G$ be labelled from $1$ to~$|G|$.
A path~$P$ of length~$l-1$ is considered to be an $l$-tuple, $(i_1,i_2,\dots,i_{l})$, where $i_1,\dots, i_l$ are distinct.
Note that $P$ is directed, so we treat $(1,2, \dots, l)$ and $(l, l-1, \dots, 1)$ differently.
Similarly, a cycle of length~$l$ is considered to be an $(l+1)$-tuple, $(i_1,i_2,\dots,i_{l+1})$ with $i_1 = i_{l+1}$, where again $i_1,\dots, i_l$ are distinct. 
Given a p.c. path $P = (i_1,i_2,\dots,i_{l})$ and $1 \le j \le l $, define $N^c(i_j;P)$ to be a colour neighbourhood of $i_j$ chosen such that both $i_{j-1}$ and $i_{j+1}$ (if they exist) belong to~$N^c(i_j;P)$.
Again, there is still a choice on $N^c(i_j;P)$, which we will specify later.
In other words, given a p.c. path $P = (i_1,i_2,\dots,i_{l})$, the neighbours of $i_j$ in $P$ are always in $N^c(i_j;P)$ for $1 \le j \le l$.

Given a p.c. path $P= (i_1,i_2,\dots,i_{l})$, $N^c(i_{1};P)$ and $N^c(i_{l};P)$, we say that $P$ has a \emph{crossing with respect to  $N^c(i_{1};P)$ and $N^c(i_{l};P)$} if there exist $1 \le a < b \le l $ such that $i_{a} \in N^c(i_{l};P)$ and $i_b \in N^c(i_1;P)$.
If $i_j \in N^c(i_{l};P)$ and $c(i_{j-1},i_j) \ne c(i_j,i_{l})$, then $P' = (i_1,i_2, \dots, i_j , i_{l} , i_{l-1},\dots, i_{j+1})$ is also a p.c. path.
It is called a \emph{rotation of $P$ with endpoint~$i_1$ and pivot point~$i_j$}.
A \emph{reflection} of $P$ is simply the p.c. path $(i_{l}, i_{l-1}, \dots,i_1)$.
The set of p.c. paths that can be obtained by a sequence of rotations and reflections of $P$ is denoted by~$\mathcal{R}(P)$.
We say $P$ is \emph{extensible} if there exists a vertex $j \notin V(P)$ such that $(i_1, \dots, i_l,j)$ or $(j,i_1, \dots, i_l)$ is a p.c. path.
This implies that if $P$ is not extensible, then $N^c(i_1;P), N^c(i_{l};P) \subseteq V(P)$ for every choice of $N^c(i_1;P)$ and $N^c(i_{l};P)$.
If every $P' \in \mathcal{R}(P)$ is not extensible, then $P$ is said to be \emph{maximal}.
Hence, all maximal and non-extensible paths have lengths at least~$\delta^c(G)$.
We now study some basic properties of a p.c path $P$ below.

\begin{lma} \label{lma:simplecycle}
Let $c$ be an edge-colouring of a graph~$G$.
Let $P=(1,2, \dots,l)$ be a properly coloured path.
Suppose that there does not exist a properly coloured cycle spanning $G[V(P)]$.
Let $a \in N^c(1;P) \setminus \{2\}$ and $b \in N^c(l;P) \setminus \{l-1\}$ be such that $b < a$, $c(1,a) \ne c(a,a+1)$ if $a <l$, and $c(l,b) \ne c(b,b-1)$ if $b >1$.
Then, $C = (1, 2, \dots, b, l, l-1, \dots, a,1)$ is a properly coloured cycle, see Figure~\ref{fig:C0}.
\begin{figure}[tbp]
\begin{center}
\includegraphics[scale=0.6]{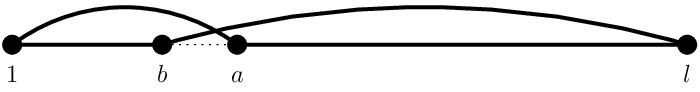}
\end{center}
\caption{Cycle $(1, 2, \dots, b, l, l-1, \dots, a,1)$}
\label{fig:C0}
\end{figure}
\end{lma}

\begin{proof}
Since $a \in N^c(1;P)$ and $b \in N^c(l;P)$, $c(1,2) \ne c(1,a)$ and $c(l,l-1) \ne c(l,b)$.
If $1 < b < a < l$, then $C$ is a p.c. cycle.
So we may assume that $b = 1$.
Moreover, if $c(l,1) \ne c(1,2)$, then $(1,2,\dots, l,1)$ is a p.c. cycle contradicting the assumption on~$G[V(P)]$.
Thus, we may assume that $c(l,1) = c(1,2)$.
Since $a \in N^c(1;P) \setminus \{2\}$, $c(1,2) \ne c(1,a)$ and so $a < l$.
Note that $c(l,l-1) \ne c(1,l) = c(1,2) \ne c(1,a) \ne c(a,a+1)$ and so $C = (1, l, l-1, \dots, a,1)$ is a p.c. cycle as required.
\end{proof}

\begin{figure}[tbp]
\begin{center}
\includegraphics[scale=0.6]{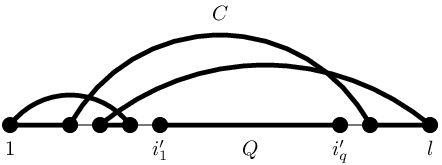}
\end{center}
\caption{A partition of $G[V(P)]$ into $C$ and $Q$ satisfying conditions $(i) - (iv)$ of Lemma~\ref{lma:keyproperties} (or Lemma~\ref{lma:crossing}), provided the graph is p.c..}
\label{fig:example}
\end{figure}

Given a p.c. path $P = (1,2, \dots, l)$, $N^c(1;P)$ and $N^c(l;P)$, we define vertices
\begin{align*}
r & = r(P) =  \min \{ b \in N^c(l;P)\},\\
s & = s(P) = \max \{ s' \in N^c(l;P) : c(l,b) = c(b,b+1) \text{ for every $b \in N^c(l;P) \cap [s']$}\},\\
u & = u(P) =  \max \{ u' \in N^c(1;P) \setminus \{l\} : c(1,a) = c(a,a+1) \text{ for every $a \in N^c(1) \cap [s+1,u']$} \},\\
w & = w(P) = \min \{ a \in N^c(1;P) \cap [u+1, l]\}.
\end{align*}
Note that the vertices $s,u, w$ may not exist for arbitrary $P$, $N^c(1;P)$ and $N^c(l;P)$.
If $s$ exists, then we further define the vertex set $S = S(P)$ to be $N^c(l;P) \cap [s]$.
Hence, $r,s,u,w$ and $S$ are functions of $P$ and its colour neighbourhoods.
In the following lemma, we show that $r, s, u$ and $w$ exist for some special p.c. paths $P$ that have crossings (with respect to $N^c(1;P)$ and $N^c(l;P)$).

\begin{lma} \label{lma:keyproperties}
Let $c$ be an edge-colouring of a graph~$G$ with $\delta^c(G) \ge d \ge 2$.
Let $P=(1,2, \dots,l)$ be a properly coloured path in $G$ that is not extensible.
Suppose there exist $N^c(1;P)$ and $N^c(l;P)$ such that $P$ has a crossing with respect to $N^c(1;P)$ and $N^c(l;P)$.
Furthermore, suppose that there does not exist a properly coloured subgraph of $G[V(P)]$ consisting of a cycle $C$ and a path $Q$ such that
\begin{enumerate}
	\item[\rm (i)] $C = (i_1,i_2 \dots, i_p,i_1)$ with $p \ge d+1$;
	\item[\rm (ii)] $Q = (i'_1,i'_2 \dots, i'_q)$, where $Q$ may be empty or consisting of a single vertex;;
	\item[\rm (iii)] $V(C) \cap V(Q) = \emptyset$ and $V(P) = V(C) \cup V(Q)$;
	\item[\rm (iv)] if $|Q| \ge 2$, then there exists $j \in [p]$ with $(i'_1,i_j) \in E(G)$ and $c(i'_1,i'_2) \ne c(i'_1,i_j)$.
\end{enumerate}
(See Figure~\ref{fig:example} for an example a partition of $G[V(P)]$ into $C$ and $Q$.)
Then $r = r(P)$ and $s = s(P)$ exist with $1 \le r \le s \le l-2$.
Moreover, the following statements hold:
\begin{description}
    \item[\rm (a)] $c(b,l) = c(b,b+1) \ne c(l, l-1)$ for all $b \in S(P)$;
	\item[\rm (b)] if $b = \min \{b' \in N^c(l;P) \setminus S\}$, then $c(b,l) \ne c(b,b+1)$;
	\item[\rm (c)] $c(1,a) = c(a,a+1) \ne c(a,a-1)$ for all $a \in ([r+1,s] \cap N^c(1;P)) \setminus \{2\}$.
\end{description}
Furthermore, if $s \ge 2$, then $u = u(P)$ and $w=w(P)$ exist with $1 \le r \le s< u < w \le l$.
In addition, the following statements hold:
\begin{description}
	\item[\rm (d)] $c(1,a) = c(a,a+1) \ne c(a,a-1)$ for all $a \in [s+1,u] \cap N^c(1;P)$;
	\item[\rm (e)] if $a \in N^c(1;P)$ and $a< w $, then $a\le u$;
	\item[\rm (f)] $c(1,w) \ne c(w,w+1)$ if $w< l$ and $c(1,w) = c(l,l-1)$ if $w =  l$.
\end{description}
\end{lma}

\begin{proof}
Write $N^c(1) = N^c(1;P)$ and $N^c(l) = N^c(l;P)$.
Since $P$ is not extensible, $N^c(1) \cup N^c(l) \subseteq V(P)$ and $l \ge d+1$.
If $c(l,r) \ne c(r,r+1)$, then $C=(r, r+1, \dots,l,r)$ is a p.c. cycle containing $N^c(l) \cup \{l\}$ of length at least $d+1$.
In addition, $Q = (r-1,r-2,\dots, 1)$ is a p.c. path, which contradicts the assumption on~$G[V(P)]$.
Hence, $c(l,r) = c(r,r+1)$ and so $s$ exists with $r \le s$.

First, we prove~(c).
Suppose that (c) is false, so there exists $a \in (N^c(1) \cap [r+1,s] ) \setminus \{2\}$ such that $c(1,a) \ne c(a,a+1)$.
Let $b \in S(P) \subseteq N^c(l)$ be maximal such that $ b<a$.
Note that $a >2$ and $1 \le r \le b < a \le s \le l-1$.
By the definition of $s$, we have $c(b,l) = c(b,b+1)$.
If $b >1$, then $c(b,l) = c(b,b+1) \ne c(b,b-1)$ as $P$ is p.c..
By Lemma~\ref{lma:simplecycle}, $C = (1,2,\dots,b,l,l-1,\dots,a,1)$ is a p.c. cycle of length at least $d+1$ as $C$ contains $N^c(l) \cup \{l\}$.
Note that $Q = (b+1, b+2, \dots, a-1)$ is a p.c. path contradicting the assumption on~$G[V(P)]$.
Thus (c) holds.

Next, we are going to show that $s \le l-2$.
Let $a' \in N^c(1)$ be maximal.
Recall that $\delta^c(G) \ge 2$, so $a' >2$.
Note that $c(1,a') = c(a', a'-1)$ or else the cycle $C=(1, 2, \dots,a',1)$ and path $Q = (a'+1,\dots, l)$ are both p.c. contradicting the assumption on~$G[V(P)]$.
Recall that $P$ has a crossing, so $r < a'$.
By (c) we have $s < a'$.
If $ s= l-1$, then $a' = l$.
This implies that $c(1,l) = c(l,l-1)$ and so $1 \notin N^c(l)$.
Let $b \in N^c(l)$ be maximal with $b < l-1 = s$.
Clearly, $b \ge 2$.
By the definition of~$s$, $c(l,b) = c(b, b+1) \ne c(b, b-1)$.
Also, $c(1,l) = c(l,l-1) \ne c(l,b)$ as $b \in N^c(l) \setminus \{l-1\}$.
Therefore, the cycle $C = (1, \dots, b, l, 1)$ is p.c. with 
\begin{align*}
|C| \ge |\{1,l\} \cup ( N^c(l) \setminus \{l-1\}) | \ge |\{1,l\}|+ | N^c(l) \setminus \{l-1\} |\ge d+1.
\end{align*}
This is a contradiction by setting $Q = (b+1, \dots, l-1)$. 
Therefore, $r \le s \le l-2$ as required.
Hence, (a) and (b) easily follow from the definitions of $s$ and $S(P)$.

Now assume that $s \ge 2$.
Recall that if $a' = \max \{a \in N^c(1)\}$, then $c(1,a') = c(a', a'-1)$ and $s <a'$.
Let $a'' \in N^c(1)$ be minimal such that $a''>s \ge 2$. 
Clearly, $a''$ exists.
Suppose that $a'' = l$ or $c(1,a'') \ne  c(a'',a''+1)$ if $a'' < l$.
By (a) and Lemma~\ref{lma:simplecycle} taking $a = a''$ and $b=s$, the cycle $C = (1,2,\dots,s,l,l-1,\dots,a'',1)$ and the path $Q = (s+1, s+2, \dots, a''-1)$ are p.c..
Moreover, $C$ contains $ N^c(1) \cup \{1\}$, so $|C| \ge d+1$, a contradiction.
Thus, we have $c(1,a'') = c(a'',a''+1)  \ne c(a'',a''-1)$ and $a'' < l$.
Therefore $u$ exists and so (d) is true by the definition of~$u$.
Furthermore, $N^c(1) \not\subseteq [u]$, where we recall that $c(1,a') = c(a', a'-1)$ with $a' = \max \{a \in N^c(1)\}$.
Thus, $w$ exists and so (e) and (f) follow.
This completes the proof of the lemma.
\end{proof}

\section{Maximal p.c. paths with crossings} \label{sec:2d+1}

Let $G$ be an edge-coloured graph with $\delta^c(G) \ge 2$.
In this section, we show that for every maximal p.c. path $P$ that has a crossing, there exists a p.c. cycle of length at least~$\delta^c(G)+1$ unless $|P| \ge 2\delta^c(G)+1$.
We consider the cases when $\delta^c(G)=2$ and $\delta^c(G) \ge 3$ separately.

\begin{lma} \label{lma:crossing2}
Let $c$ be an edge-colouring of a graph~$G$ such that $\delta^c(G) \ge 2$.
Let $P = (1, 2, \dots, l)$ be a maximal properly coloured path.
Suppose $P$ has a crossing with respect to some $N^c(1;P)$ and $N^c(l;P)$.
Then, there exists a properly coloured cycle $C$ in $G[V(P)]$.
\end{lma}

\begin{lma} \label{lma:crossing}
Let $c$ be an edge-colouring of a graph~$G$ such that $\delta^c(G) = d \ge 3$.
Let $P=(1, 2, \dots, l)$ be a maximal properly coloured path.
Suppose $P$ has a crossing with respect to some $N^c(1;P)$ and $N^c(l;P)$.
Then, there exists a properly coloured subgraph of $G[V(P)]$ consisting of a cycle $C$ and a path $Q$ such that
\begin{enumerate}
	\item[\rm (i)] $C = (i_1,i_2 \dots, i_p,i_1)$ with $|C| =p  \ge d$;
	\item[\rm (ii)] $Q = (i'_1,i'_2 \dots, i'_q)$, where $Q$ may be empty or consisting of a single vertex;
	\item[\rm (iii)] $V(C) \cap V(Q) = \emptyset$ and $V(P) = V(C) \cup V(Q)$;
	\item[\rm (iv)] if $|Q| \ge 2$, then there exists $j \in [p]$ with $(i'_1,i_j) \in E(G)$ and $c(i'_1,i'_2) \ne c(i'_1,i_j)$.
\end{enumerate}
(See Figure~\ref{fig:example} for an example a partition of $G[V(P)]$ into $C$ and $Q$.)
Moreover, if $|P| \le 2d$, then $|C| = p \ge d+1$.
\end{lma}

In Lemma~\ref{lma:crossing2}, i.e. when $\delta^c(G) =2$, we only show the existence of a p.c. cycle $C$ in $G[V(P)]$.
In Lemma~\ref{lma:crossing}, i.e. when $\delta^c(G) \ge 3$, we further show that there is a spanning p.c. path~$Q$ in $G[V(P) \setminus V(C)]$, if $V(P) \setminus V(C) \ne \emptyset$, where $C$ is a p.c. cycle of length at least $\delta^c(G)$.
Next, we show that Lemma~\ref{lma:crossing2} and Lemma~\ref{lma:crossing} imply Theorem~\ref{thm:2d+1}.

\begin{proof}[Proof of Theorem~\ref{thm:2d+1}]
Let $P$ be a maximal p.c. path in $G$.
Without loss of generality, $P=(1,2, \dots, l)$.
Fix $N^c(1;P)$ and $N^c(l;P)$.
If $P$ has a crossing, then we are done by Lemma~\ref{lma:crossing2} and Lemma~\ref{lma:crossing}.
If $P$ does not have a crossing, then $|(N^c(1;P) \cup \{1\})\cap (N^c(l;P)\cup \{l\})| \le 1$ and so $|P| \ge |(N^c(1;P) \cup \{1\})\cup (N^c(l;P)\cup \{l\})| \ge 2d+1$.
\end{proof}

First we prove Lemma~\ref{lma:crossing2}, that is the case when $\delta^c(G) \ge 2$.
\begin{proof}[Proof of Lemma~\ref{lma:crossing2}]
Suppose the lemma is false.
Let $G$ be a graph with an edge-colouring~$c$ containing a maximal p.c. path $P = (1,2, \dots, l)$ that contradicts Lemma~\ref{lma:crossing2}, so $\delta^c(G) \ge 2$.
Fix $N^c(i;P)$ for each $i \in [l]$ such that $P$ has a crossing.
Note that $|N^c(i;P)\cap V(P)| \ge 2$ for all $i \in[l]$.
The induced subgraph $H = G[V(P)]$ does not contain any p.c. cycle and $\delta^c(H) \ge 2$.
By Theorem~\ref{thm:Yeo}, there exists a vertex~$z$ in $V(H)$ such that no connected component of~$H-z$ is joined to $z$ with edges of more than one colour.
However, $H$ is 2-connected as $P$ has a crossing.
This contradicts the existence of such~$z$ as $\delta^c(H) \ge 2$.
\end{proof}

Here, we sketch the proof of Lemma~\ref{lma:crossing}.
Let $G$ be a graph with an edge-colouring~$c$.
Let $P = (1,2, \dots, l)$ be a maximal p.c. path in~$G$.
Let $C$ be a p.c. cycle such that 
\begin{align}
	V(C) \subseteq [l] \text{ and } [l] \setminus V(C) \text{ is an interval} \label{eqn:C-keyproperty},
\end{align}
$[l] \setminus V(C) = [i'_1, i'_q]$ say.
Without loss of generality (by considering $(i'_q,i'_{q-1}, \dots,i'_1)$ instead if necessary), we may assume that $i'_1 \notin \{1, l\}$, so $N^c(i'_1;P) \cap V(C) \ne \emptyset$.
By setting $Q = (i'_1,i'_1+1, \dots  i'_q)$, $C$ and $Q$ satisfy properties (ii)--(iv) in Lemma~\ref{lma:crossing}.
Thus, to prove Lemma~\ref{lma:crossing}, it suffices to show that there exists a p.c. cycle $C$ satisfying~\eqref{eqn:C-keyproperty} with $|C| \ge d$.
Suppose the lemma is false.
Let $P$ be a maximal p.c. path in $G$ contradicting the lemma.
Apply Lemma~\ref{lma:keyproperties} and obtain vertices $r$, $s$, $u$ and $w$ with $1 \le r \le s  <u <w \le l$.
Then, $C_0 = (1,2,\dots,s,l,l-1,\dots,w,1)$ is a p.c. cycle by Lemma~\ref{lma:simplecycle} satisfying~\eqref{eqn:C-keyproperty}, see Figure~\ref{fig:C1}.
By further assuming that $|S(P)| \ge |S(P')|$ for all $P' \in \mathcal{R}(P)$, we then deduce that $|C_0| \ge d$, Claim~\ref{clm:|C|}.
Thus, Lemma~\ref{lma:crossing} holds with $ p \ge d$.
A detailed analysis of $N^c(i;P)$ for all $i \in [l]$ is needed in order to prove the `moreover' statement, that is if $|P| \le 2d$ then $|C| = p \ge d+1$.

\begin{proof}[Proof of Lemma~\ref{lma:crossing}]
Suppose the lemma is false.
Let $G$ be an edge-minimal graph with an edge-colouring~$c$ containing a maximal p.c. path $P = (1,2, \dots, l)$ that contradicts Lemma~\ref{lma:crossing}.
By the discussion above, in order to prove Lemma~\ref{lma:crossing}, it suffices to show that there exists a p.c. cycle $C$ satisfying~\eqref{eqn:C-keyproperty} with $|C| \ge d$.
Furthermore, if we can show that $|C| \ge d+1$, then the `moreover' statement of the lemma also holds.
Fix $N^c(1;P)$ and $N^c(l;P)$ such that $P$ has a crossing.
Let 
\begin{align*}
A  = \{a_i : 1 \le i \le d_1 \}= N^c(1;P) \text{ and } B = \{b_j : 1 \le j \le d_l\} = N^c(l;P),
\end{align*}
where both $(a_i)_{i = 1}^{d_1}$ and $(b_j)_{j=1}^{d_l}$ are increasing sequences.
By maximality of $P$ and choices of $N^c(i;P)$, we have
\begin{align*}
 d^c(1) = d_1,\  d^c(l) = d_l, \ a_1 = 2 \text{ and } b_{d_l} = l-1.
\end{align*}
If $l \in A$ and $1 \in B$, then $(1,2, \dots, l, 1) $ is a cycle of length $l \ge d+1$ as $A \cup \{1\} \subseteq V(P) = [l]$.
Hence, 
\begin{align}
l \notin A \text{ or }1 \notin B, \label{eqn:lA1B}
\end{align}
so
\begin{align}
l\ge d+2. \label{eqn:lA1B2}
\end{align}
If $l-1 \in A$ and $c(1,l-1) \ne c(l-1,l-2)$, then $(1,2, \dots, l-1,1)$ is a p.c. cycle of length $l-1 \ge d+1$.
Therefore, 
\begin{align}
\text{if $l-1 \in A$, then $c(1,l-1) = c(l-1,l-2)$.} \label{eqn:l-1A}
\end{align}
Apply Lemma~\ref{lma:keyproperties} and obtain $r = r(P)$ and $s = s(P)$. 
Note that $r,s \in B$ and $r = b_1$.
Recall that $S = S(P) = [r,s] \cap B$.
Also, notice that
\begin{align}
	c(1, a_{d_1}) = c(a_{d_1}, a_{d_1}-1) \label{eqn:c(1,a_p)}
\end{align}
or else $(1,2, \dots, a_{d_1},1)$ is a p.c. cycle satisfying~\eqref{eqn:C-keyproperty} with $|C| \ge d+1$ as $A \cup \{1\} \subseteq V(C)$.
We further assume that $|S| = |S(P)| \ge |S(P')|$ for all $P' \in \mathcal{R}(P)$, i.e. $|S|$ is maximal.

If $|S| \ge 2$, then $s \ge 2$.
If $|S| = 1$ and $s = 1$, then $l \notin A$ by~\eqref{eqn:lA1B}.
Let $P'= (l, l-1, \dots, 1)$ be the reflection of $P$.
Set $N^c(1;P') = N^c(1;P) = A$ and $N^c(l;P') = N^c(l;P)=B$.
Note that $P'$ has a crossing (with respect to $N^c(1;P')$ and $N^c(l;P')$) and $P' \in \mathcal{R}(P)$.
By Lemma~\ref{lma:keyproperties}, $S(P')$ is non-empty as $r(P') \in S(P')$.
Since $|S|$ is maximal and $|S| =1$, we deduce that $|S(P')| = 1$.
Moreover, $r(P') \ne l$ as $l \notin A = N^c(1;P')$.
Therefore, by replacing $P$ with $P'$ and relabelling the vertices, we may assume that $s \ge 2$.
In summary, without loss of generality, we may assume that $s \ge 2$.
Apply Lemma~\ref{lma:keyproperties} and obtain $u = u(P)$ and $w = w(P)$ with 
\begin{align}
2,r \le s < u < w \le l. \label{eqn:1}
\end{align}
In the next claim, we find a p.c. cycle $C_0$ satisfying~\eqref{eqn:C-keyproperty} with $|C_0| \ge d $.
Hence, this completes the proof of Lemma~\ref{lma:crossing} without the `moreover' statement.

\begin{figure}[tbp]
\begin{center}
\includegraphics[scale=0.6]{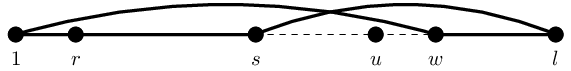}
\end{center}
\caption{$C_0 = (1,2,\dots,s,l,l-1,\dots,w,1)$} 
\label{fig:C1}
\end{figure}

\begin{clm} \label{clm:|C|}
The following statements are true:
\begin{description}
\item{\rm (a)} $C_0 = (1,2,\dots,s,l,l-1,\dots,w,1)$, see Figure~\ref{fig:C1}, is a p.c. cycle satisfying~\eqref{eqn:C-keyproperty}.
\item{\rm (b)} $|C_0| = d = |A|$, $l \le 2d$ and $S = [r,s]$.
Moreover, every $P'' \in \mathcal{R}(P)$ has a crossing independent of the choices of the colour neighbourhoods.
\item{\rm (c)} Let $t = t(P) =  u - |S| +1 \ge 3$. Then
\begin{align}
 A & = \left\{ \begin{array}{ll}
	[2,r] \cup [t,u] \cup [w,l]	& \text{if $r \ne 1$,}\\
	\{ 2 \} \cup [t,u] \cup [w,l-1]	& \text{if $r =1$,}
    \end{array} \right. 
    \nonumber
\end{align}
where all intervals are non-empty and pairwise disjoint.
\item{\rm (d)} $c(1,i) = c(i,i+1)$ for every $t \le i \le u$;
\item{\rm (e)} Given an integer $b$ with $r \le b \le  s$, the path $P^* = (b+1,b+2, \dots, l,b,b-1,\dots,1)$ is p.c. and a member of $\mathcal{R}(P)$. 
Moreover, if $ b <t$ and 
\begin{align*}
 N^c(1;P^*)  & = \left\{ \begin{array}{ll}
	A	& \text{if $b \ne 1$,}\\
	A \cup \{l\} \setminus \{2\} & \text{if $b =1$,}
    \end{array} \right.
\end{align*}
then $S(P^*) = [t,u]$.

\item{\rm (f)} If $P''=(i_1, \dots, i_l) \in \mathcal{R}(P)$ with $s(P'') \ne i_1$ and $|S(P'')| = |S|$ with respect to some $N^c(i_1;P'')$ and $N^c(i_l;P'')$, then the corresponding statements of {\rm(a)--(e)} hold (by the map $i_j \rightarrow j$ and recall that $r,s,t,u,v$ are functions of $P$).
\end{description}
\end{clm}

\begin{proof}[Proof of claim]
Note that $2 \le s < l-1$ and $w \ge 4$ by~\eqref{eqn:1}.
Hence (a) follows from Lemma~\ref{lma:simplecycle} and Lemma~\ref{lma:keyproperties}~(a) and~(f).
Recall that $r \in S $, so $c(l,l-1) \ne c(l,r) = c(r,r+1)$ by Lemma~\ref{lma:keyproperties}~(a).
In addition, if $r \ne 1$, then $c(r,l) \ne c(r,r-1)$ since $P$ is p.c..
Thus, the path 
\begin{align*}
P' = (r+1,r+2, \dots, l,r,r-1,\dots,1)
\end{align*}
(see Figure~\ref{fig:Pprime}), which is obtained by a rotation with endpoint~$1$ and pivot point~$r$ followed by a reflection, is p.c. and so a member of~$\mathcal{R}(P)$.
\begin{figure}[tbp]
\begin{center}
 \includegraphics[scale=0.6]{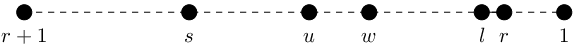}
\end{center}
\caption{$P' = (r+1,r+2, \dots, l,r,r-1,\dots,1)$} \label{fig:Pprime}
\end{figure}
Since $c(r+1,r) \ne c(r+1,r+2)$, we choose $N^c(r+1;P^{\prime})$ such that $r \in N^c(r+1;P^{\prime})$.
Set $N^c(1;P^{\prime}) = A$ if $r \ne 1$ and $N^c(1;P^{\prime}) = ( A \cup \{l\} ) \setminus \{2\}$ if $r = 1$.
Since $u \in A$ and $2 \le s < u < l$ by~\eqref{eqn:1}, $u \in N^c(1;P^{\prime})$. 
Thus, $P'$ has a crossing.
If $a \in ( A \cap [r+1,u] ) \setminus \{2\} $, then $c(1,a) = c(a,a+1)$ by Lemma~\ref{lma:keyproperties}~(c) and~(d).
Hence, $S(P')$ contains $( A \cap [r+1,u] ) \setminus \{2\}$.
Since $|S|$ is maximal, 
\begin{align}
|S| \ge |S(P^{\prime})| \ge | A\cap[r+1,u]| - \delta_{1,r} \label{eqn:|S'|},
\end{align}
where $\delta_{1,r} =1$ if $r=1$, and $\delta_{1,r}=0$ otherwise.
Recall \eqref{eqn:lA1B} that $1 \notin B$ or $l \notin A$, so $|\{l\}\setminus A| - \delta_{1,r} \ge 0$.
Note that 
\begin{align}
 |C_0| = & |[1,s]\cup [w,l]|  = |A| + 1+|[2,s]\setminus A| +|[w,l] \setminus A| - |[s+1,u] \cap A|. \label{eqn:|C|}
\end{align}
By adding \eqref{eqn:|S'|} and~\eqref{eqn:|C|} together, we have
\begin{align}
 |C_0| + |S| & \ge  |A| + 1+|[2,r]\setminus A|+|[w,l] \setminus A|- \delta_{1,r} +|[r+1,s]| \nonumber \\
	& \ge  |A| + 1+ |\{l\} \setminus A|- \delta_{1,r} +|S\setminus \{r\}|  \nonumber\\
	& \ge  |A| + |S| \ge d +|S|. \label{eqn:|S|+|C|}
\end{align}
This implies $|C_0| \ge d$.
Since $C_0$ satisfies~\eqref{eqn:C-keyproperty}, we have $|C_0| =d$ and $l \le 2d$ (or else $G$ is not a counterexample).
Therefore, all inequalities in~\eqref{eqn:|S|+|C|} are actually equalities.
Moreover, we deduce that
\begin{align}
S = [r,s], \qquad |A|=d \qquad \text{and} \qquad [2,r]\cup [w,l-\delta_{1,r}] \subseteq A. \label{eqn:3.3}
\end{align}
Let $P'' = (i_1, \dots, i_l) \in  \mathcal{R}(P)$.
Recall that $P''$ is a maximal p.c. path as is $P$, so $N^c(i_1;P''), N^c(i_l;P'') \subseteq [l]$ for any choice of $N^c(i_1;P'')$ and $N^c(i_l;P'')$.
Since $|P''| = l \le 2d$ and $|N^c(i_1;P'')|, |N^c(i_l;P'')| \ge \delta^c(G) = d$, $P''$ has a crossing for all choices of $N^c(i_1;P'')$ and $N^c(i_l;P'')$.
So (b) holds.

Note that equality also holds in~\eqref{eqn:|S'|}, so $|S| = |S(P^{\prime})| = |A \cap [r+1,u]|- \delta_{1,r}$.
Recall the sentence above~\eqref{eqn:|S'|} that $S(P^{\prime})$ contains $( A \cap [r+1,u] ) \setminus \{2\}$.
Thus we have 
\begin{align}
S(P^{\prime}) = ( A \cap [r+1,u] ) \setminus \{2\} \text{ and } |S(P')| = |S|. \label{eqn:SP'}
\end{align}
Next we are going to show that $S(P^{\prime})$ is also an interval.
If $|S|=1$, then there is nothing to prove.
If $|S| \ge 2$, then $s(P^{\prime}) \ne r+1$.
Apply Lemma~\ref{lma:keyproperties} and obtain $u(P')$ and $ w(P')$.
By Claim~\ref{clm:|C|}~(a) and~(b) taking $P = P'$, we deduce that $S(P^{\prime})$ is an interval.
Hence, $S(P^{\prime}) = [t,u]$ and so (d) holds.
Recall that $A \cap [u+1, w-1] = \emptyset$ by Lemma~\ref{lma:keyproperties}~(e).
Therefore (c) follows from~\eqref{eqn:3.3} and~\eqref{eqn:SP'}.

Next, are going to prove (e).
If $b =r$, then $P^* = P'$ as defined above and so (e) holds.
Suppose that $r < b \le s$.
Since $b \in S$, so $c(l,l-1) \ne c(l,b) = c(b,b+1)$ by Lemma~\ref{lma:keyproperties}~(a).
So $P^*$ is a p.c. path.
Note that $P^*$ is obtained from $P$ by a rotation with endpoint~$1$ and pivot point~$b$ followed by a reflection, so $P^* \in \mathcal{R}(P)$.
Further assume that $b < t$.
Set $N^c(1;P^*) = A$.
We get $[t,u] \subseteq S(P^*)$ by~(d).
By the maximality of $|S|$, $S(P'') = [t,u]$ and so (e) holds.

To prove (f), note that since $s(P'') \ne i_1$, $u(P'')$ and $w(P'')$ exist by Lemma~\ref{lma:keyproperties}.
Hence, (f) follows from (a)--(e) taking $P = P''$.
This completes the proof of the claim.
\end{proof}

Next, we show that $|S| \ge 2$.

\begin{clm} \label{clm:|S|>2}
$|S| \ge 2$ and $d \ge 4$.
\end{clm}

\begin{proof}[Proof of claim]
If $|S| \ge 2$, then Claim~\ref{clm:|C|}~(c) implies that $d \ge 4$ as all intervals are disjoint and non-empty.
Hence, to prove the claim, it suffices to show that $|S| \ge 2$.
Suppose $|S|=1$.
Recall that we have assumed that $s \ge 2$, so $r =s \ge 2$.
Thus, 
\begin{align}
A =	[2,r] \cup \{u\} \cup [w,l] \label{eqn:|S|>2}
\end{align}
by~Claim~\ref{clm:|C|}~(c), where the intervals are disjoint.
Note that $c(1,l) = c(l,l-1)$ by~\eqref{eqn:c(1,a_p)}.
If $l - 1 \in A$, then $c(1,l-1) =  c(l-1,l-2)$ by~\eqref{eqn:l-1A}.
By taking $P' = (l,l-1, \dots, 1)$, $N^c(1;P') = A$ and $N^c(l;P') = B$, we have $\{ l,l-1\} \subseteq S(P')$.
This implies that $|S(P')| \ge 2$ contradicting the maximality of $|S|$.
Thus, $l-1 \notin A$ and so \eqref{eqn:|S|>2} become $ A =	[2,r] \cup \{u,l\}$.
By Lemma~\ref{lma:keyproperties}~(d), $c(1,u) = c(u,u+1) \ne c(u,u-1)$ and so $(1, 2, \dots, u,1)$ is a p.c. cycle satisfying~\eqref{eqn:C-keyproperty}.
If $u \ge d+1$, then Lemma~\ref{lma:crossing} holds.
Note that $|A| =d$ by Claim~\ref{clm:|C|}~(b) and so we have $u = d$.
Thus, $A = [2,d] \cup \{l\}$.
By Claim~\ref{clm:|C|}~(d),
\begin{align}
c (1,d) = c(d,d+1) \ne c(d,d-1)  \label{eqn:|S|>2b}.
\end{align}
By \eqref{eqn:|S|>2}, $r=d-1$ and so Lemma~\ref{lma:keyproperties}~(a) implies that
\begin{align}
c (l,d-1) = c(d,d-1) \ne c(d-1,d-2).  \label{eqn:|S|>2a}
\end{align}
Let $b \in B \setminus \{d-1,l-1\} \ne \emptyset$ as $d \ge 3$, so $b \in [d,l-2]$.
Also, recall that $a_{d_1}$ is the largest integer in $A$.
Since $ l = a_{d_1}$, \eqref{eqn:c(1,a_p)} implies that $c(1,l) = c(l,l-1) \ne c(l,b)$. 
If $c(b,l) \ne c(b,b-1)$, then $(1,2, \dots, b, l,1)$ is a p.c. cycle of length at least $b +1 \ge d+1$ satisfying~\eqref{eqn:C-keyproperty}, a contradiction. 
Therefore, $c(b,l) = c(b,b-1)$ for all $b \in B \setminus \{d-1,l-1\}$.
If $b = d$, then $c(d,l) = c(d,d-1) = c (l,d-1)$, where the last equality is due to~\eqref{eqn:|S|>2a}.
However, this contradicts the fact that $d-1, d \in B = N^c(l;P)$. 
If $b = d+1$, then $c(d+1,l) = c(d+1,d) = c(1,d)$, where the last equality is due to~\eqref{eqn:|S|>2b}.
Hence, $(1,d,d+1,l)$ is a monochromatic path of length~3.
Let $G'$ be the edge-coloured subgraph of $G$ obtained by removing the edge $(d,d+1)$.
Note that $\delta^c(G')=d$.
By \eqref{eqn:|S|>2b} and~\eqref{eqn:|S|>2a}, the path $$P'=
(d+1, d+2, \dots, l , d-1, d-2, \dots, 1, d)$$ is p.c. in $G$.
Moreover, $P'$ is obtained by a rotation of~$P$ with pivot point~$d-1$ and endpoint~$1$ followed by a rotation with pivot point~$d$ and endpoint~$d$.
Hence, $P' \in \mathcal{R}(P)$ and is maximal in~$G$.
Since $P'$ does not contain the edge $(d,d+1)$, $P'$ is also p.c. and maximal in $G'$.
However, $G'$ contradicts the edge-minimality of~$G$.
Therefore, $B \cap \{d,d+1\}=\emptyset$ and so $B \subseteq \{d-1\} \cup [d+
2,l-1]$ as $r= d-1$.
Together with $|B| \ge d$, this implies that $l \ge 2d+1$, a contradiction. 
This completes the proof of Claim~\ref{clm:|S|>2}.
\end{proof}

In the next claim, we show that if necessary $t$ (as defined in Claim~\ref{clm:|C|}~(c)) may be assumed to be at least~$r+3$.

\begin{clm} \label{clm:t-1notinS}
We may assume that $t \ge r+3$.
Moreover, $d \ge 5$.
\end{clm}

\begin{proof}[Proof of claim]
First assume that $t \ge r+3$.
Since $|S| \ge 2$, Claim~\ref{clm:|C|}~(c) implies that $t = u - |S| +1 \le u-1$.
In addition, $(1,2, \dots, t+1,1)$ is a p.c. cycle of length $t+1 \ge r+4 \ge 5$.
Moreover, this cycle satisfies~\eqref{eqn:C-keyproperty}.
Thus, we may assume that $d \ge 5$ or else the lemma holds.
Hence to prove the claim, it suffices to show that $t \ge r+3$.

By Claim~\ref{clm:|C|}~(c), we deduce that $t \ge \max\{r+1, 3\}$.
Suppose the claim is false, so either $t=r+1$ or~$t=r+2$.
Recall Claim~\ref{clm:|C|}~(b) and Claim~\ref{clm:|S|>2} that $S = [r,s]$ and $|S| \ge 2$, so $t-1 \in S$.
If $ r\ne 1$, by Claim~\ref{clm:|C|}~(e) taking $b = t-1$, we obtain a path $P^* = (t, t+1,\dots, l , t-1, t-2,\dots,1) \in \mathcal{R}(P)$.
Moreover, by setting $N^c(1; P^*) = A$, $S(P^*) = [t ,u]$.
Therefore, by replacing $P$ with $P^*$ if necessary, we may assume that $r=1$.
Hence, $t = 3$.
By Claim~\ref{clm:|C|}~(b) and~(c), we deduce that 
\begin{align}
r = 1, \quad s \ge 2, \quad t = 3, \quad u = s+2, \quad S=[1,s],   \quad A = [2 ,s+2] \cup [w,l-1]. \label{eqn:clm:t-1notinS}
\end{align}

Given a path $Q = (i_1,i_2, \dots, i_l)$, we define the path $\phi (Q)$ to be $(i_3, i_4,\dots, i_l ,i_2, i_1)$.
Set $P^0 = P = (1, 2, \dots, l)$.
For $1 \le i \le (l-1)/2$, define $P^i$ to be $\phi (P^{i-1})$ .
Thus, $P^1 = (3,4 , \dots, l , 2,1)$ and $P^2 = (4,5 , \dots, l , 2,1, 4,3)$.
We write $p^i_j$ to be the $j$th vertex of $P^i$. 
For example, $p^0_j = j$ for all $1 \le j \le l$ and $p^2_2 = 5$.
We are going to show that the following statements hold for~$ 0 \le i \le (l-1)/2$ (subject to some choices of the colour neighbourhoods which will become clear):
\begin{enumerate}
 \item[\rm (i)] $P^i \in \mathcal{R}(P)$;
 \item[\rm (ii)] $S(P^i) = \{ p^i_j : 1 \le j \le s \}$ and $|S(P^i)| = s \ge 2$. 
 In particular, $c( p^i_l,  p^i_j) =  c( p^i_j,  p^i_{j+1} )$ by Lemma~\ref{lma:keyproperties}~(a);
 \item[\rm (iii)] $N^c( p^i_1 ; P^i) =A^i$, where $A^i = \{ p^i_j : j \in A \}$;
 \item[\rm (iv)] for $3 \le j \le s+2 $, we have $c( p^i_1, p^i_j) = c(p^i_j,p^i_{j+1}) \ne c(p^i_{j},p^i_{j-1})$;
 \item[\rm (v)] $N^c( p^i_{2};P^i) = \{ p^i_{j} :  j \in \{1\} \cup [3,d+1] \}$;
 \item[\rm (vi)] for $4 \le j \le d+1$, we have $c( p^i_{2},p^i_{j}) = c(p^i_{j},p^i_{j+1}) \ne c( p^i_{j} , p^i_{j-1})$.
\end{enumerate}

First, we are going to show that (i)--(iv) hold by induction on~$i$.
There is nothing to prove when $i=0$ by~\eqref{eqn:clm:t-1notinS} and Claim~\ref{clm:|C|}~(d), so we may assume that $i \ge 1$ and (i)--(iv) hold for~$i-1$.
For simplicity, we may assume that $i=1$ by considering the map $p^{i-1}_j \mapsto j$.
By Claim~\ref{clm:|C|}~(e) taking $b = 2$, we obtain that 
\begin{align*}
P^1 = (p_1^1, p_2^1, \dots, p_l^1) = (3,4 , \dots, l , 2,1) \in  \mathcal{R}(P)
\end{align*}
and so (i) holds.
Set $N^c(p^1_l;P^1) = N^c(1;P^{0}) = A^{0}$ by~(iii).
By Claim~\ref{clm:|C|}~(e) and~(iv), we have $S(P^1) = [3,s+2] =\{ p_j^1 : j \in [s] \}$ implying~(ii).
By Claim~\ref{clm:|C|}~(f) and~(c) (with $P'' = P^1$, $r(P^1) = p^1_1$ and $|S(P^1)| = |S| = s$), there exists an integer $t(P^1) \ge3 $ such that 
\begin{align*}
N^c(  p_1^1 ; P_1) = \{ p_j^1 : j \in \{2\} \cup [t(P^1), t(P^1)+ s-1] \cup [w,l-1] \}.
\end{align*}
(Note that $w$ in the equation above is indeed the same $w$ in~\eqref{eqn:clm:t-1notinS} as $|A^1| = |A^0| =d$.)
In addition, by Claim~\ref{clm:|C|}~(f) and~(d)
\begin{align}
c(p_1^1,p_1^{j}) = c(p_1^{j},p_1^{j+1}) \ne c(p_1^{j},p_1^{j-1}) \nonumber
\end{align}
for all $t(P^1) \le j \le t(P^1)+s-1$.
If $t(P^1) > 3$, then Claim~\ref{clm:t-1notinS} is true by taking $P = P^{1}$, a contradiction.
Thus, $t(P^1) = 3$ implying that $N^c(p_1^1;P^1) =A^{1}$, so both (iii) and (iv) are true.
Therefore, (i)--(iv) hold for all~$0 \le i \le (l-1)/2$.

Next, we show that (i)--(iv) imply (v) and~(vi).
For simplicity, we may assume that $i=0$ by considering the map $p^{i}_j \mapsto j$.
By Claim~\ref{clm:|C|}~(e) taking $b = 1$, we obtain that 
\begin{align*}
P'' = (2,3,\dots,l,1) \in \mathcal{R}(P).
\end{align*}
Moreover, by taking $N^c(1;P'') = ( N^c(1;P^0) \cup \{l\} )\setminus \{2\}$, we have $S(P'')=[3,s+2]$.
Again by Claim~\ref{clm:|C|}~(f) and~(c), we conclude that 
\begin{align}
	N^c(2;P'') = \{3\} \cup [t',u'] \cup [w',l] \cup \{1\} \label{eqn:clm:t-1notinS:5}
\end{align}
for some $t' \le  u' \le w'$ with $t' = u' - s+1$.
(Here, $[w',l]$ may be an empty interval.)
Moreover, Claim~\ref{clm:|C|}~(f) and~(d) imply that
\begin{align}
c(2,j) & = c(j,j+1) \ne c(j,j-1) & \text{for $t' \le j \le u'$.}\label{eqn:clm:t-1notinS:2}
\end{align}
Recall that $2 \in S(P^0)$ by~(ii), so $c(l,2) = c(2,3)$.
This means $l \notin N^c(2;P'')$ and so \eqref{eqn:clm:t-1notinS:5} becomes
\begin{align}
 N^c(2;P'') =  \{3\} \cup [t',u']  \cup \{1\} \nonumber.
\end{align}
Note that $(2, 3, \dots, u',2)$ is a p.c. cycle by~\eqref{eqn:clm:t-1notinS:2} and satisfies~\eqref{eqn:C-keyproperty} for~$P''$.
We may assume that $u' = d+1$ or else Lemma~\ref{lma:crossing} holds.
Therefore, $N^c(2;P'') =  [3,d+1] \cup \{1\}$ and so (v) holds by setting $N^c(2;P^0) = N^c(2;P'')$.
Note that $t' = 4$ and $u' = d+1$, so (vi) is true by~\eqref{eqn:clm:t-1notinS:2}.
In summary, we have shown that (i)--(vi) hold for all $0 \le i \le (l-1)/2$.

Claim~\ref{clm:|S|>2} implies that $s = |S| \ge 2$ and $d \ge 4$.
By (iii) and (ii), we have $c(p^i_1,p^i_3) \ne c(p^i_1,p^i_2) = c(p^i_1,p^i_l)$.
Similarly, we have 
\begin{align*}
c(p^i_1,p^i_3) &\ne c(p^i_3,p^i_2) , & 
c(p^i_2, p^i_4) &\ne c(p^i_2, p^i_1), & 
c(p^i_2, p^i_4) &\ne  c(p^i_4,p^i_3)
\end{align*}
by (iv), (v) and (vi) respectively.
Note that $p^i_j = j +2i$ for $j+2i \le l$ and $p^i_l = 2i-1$ for $1 \le i \le (l-1)/2$.
Therefore, in summary, we have $ c(j-1,j) \ne c(j,j+2)  \ne c(j+1,j+2)$ for $2 \le j \le l-2$.
Set $j = l-2$, so 
\begin{align*}
c(l-3,l-2) \ne c(l-2,l)  \ne c(l-1,l).
\end{align*}
Since $l-1 \in A$ by~\eqref{eqn:clm:t-1notinS}, we have
\begin{align*}
c(1,2) \ne c(1,l-1) = c(l-2,l-1) \ne c(l-1,l),
\end{align*}
where the equality is due to~\eqref{eqn:c(1,a_p)}.
Therefore, $(1, 2,\dots, l-2,l,l-1,1)$ is a p.c. cycle spanning $[l]$.
This is a contradiction, so the claim holds.
\end{proof}

Next we show that $|S|$ is at least three. 

\begin{clm} \label{clm:S>=3}
$|S|\ge 3$.
\end{clm}

\begin{proof}[Proof of claim]
Suppose the contrary, so $|S|=2$ by Claim~\ref{clm:|S|>2}.
Without loss of generality $r \ne 1$, otherwise consider the path $P' = (2,3,\dots,l,1)$ instead with $N^c(1;P') = A \cup \{l\} \setminus \{2\}$ (as $l \notin A$ by~\eqref{eqn:lA1B}) by~Claim~\ref{clm:|C|}~(e).
Thus,
\begin{align} 
A = [2,r] \cup \{t,t+1\} \cup [w,l] \label{eqn:clm:S>=3A}
\end{align}
by Claim~\ref{clm:|C|}~(c).
It should be noted that here $t$ is not necessarily at least $r+3$.
We divide into separate cases depending on~$w$.

\noindent\textbf{Case 1: $w \le l-2$.}
Note that $c(1,l-1) = c(l-1,l-2)$ by~\eqref{eqn:l-1A}.
Let $P' = (l,l-1,\dots,1)$ be the reflection of $P$.
Set $N^c(1;P') = A$ and $N^c(l; P') = B$.
Both $l-1$ and $l$ are members of $S(P')$, so  $|S(P')|\ge 2$.
Since $|S(P')| \le |S|=2$, we have $l-2 \notin S(P')$ and so $c(1,l-2) \ne c(l-2,l-3)$ by Lemma~\ref{lma:keyproperties}~(b) taking $P = P'$.
This implies that $(1, 2,\dots, l-2,1)$ is a p.c. cycle. 
We may further assume that this cycle has length at most~$d$ as it satisfies~\eqref{eqn:C-keyproperty} and $G$ is a counterexample.
Hence, $l = d+2$ by~\eqref{eqn:lA1B2}.
Moreover, $B = [2,d+1]$ as $1, d+2 \notin B$ and so $r=2$ and $s=3$.
Since the p.c. cycle $C_0=(1,2,3,d+2,d+1,\dots, w,1)$ has length~$d$ by Claim~\ref{clm:|C|}~(a) and (b), we get $w = 6$.
As $3,4 \in B$ and $s = 3$, we have $c(4,d+2) \ne c(3,d+2) = c(3,4)$.
However, $(1, 2,3, 4,d+2,d+1\dots, 6,1)$ is a p.c. cycle of length $d+1$ satisfying~\eqref{eqn:C-keyproperty} which is a contradiction.

\noindent\textbf{Case 2: $w = l-1$.}
By~\eqref{eqn:clm:S>=3A}, we have $A = [2,r] \cup \{t,t+1,l-1,l\}$
and so $d = |A|=r+3$.
Let $P' = (l,l-1,\dots,1)$ be the reflection of~$P$.
Set $N^c(1;P') = A$ and $N^c(l; P') = B$.
Since $l-1, l \in N^c(1;P')$, we have $l-1,l \in S(P')$ by~\eqref{eqn:l-1A} and~\eqref{eqn:c(1,a_p)} respectively.
Notice that $2 \le |S(P')|\le |S| \le 2$ and so $S(P') = \{l-1, l\}$.
By applying Claim~\ref{clm:|C|}~(f) and~(c) with $P'' = P'$, we have $2 \in N^c(l ; P') = B$.
Thus, $r=2$ by~\eqref{eqn:lA1B} (as $l \in A$) and so $s=3$.
Therefore, 
\begin{align}
A = \{2,t,t+1,l-1,l\} \label{eqn:clm:S>=3A3}
\end{align}
and $d=5$.
By Claim~\ref{clm:|C|}~(d), $c( 1,t+1) = c(t+1,t+2) \ne c(t+1,t)$ as $t+1 \le u$.
This implies that $(1,2,\dots,t+1,1)$ is a p.c. cycle satisfying~\eqref{eqn:C-keyproperty}. 
Hence, either $t =3$ or $t=4$ or Lemma~\ref{lma:crossing} holds.

First suppose that $t = 3$.
By Claim~\ref{clm:|C|}~(e) with $b = 2$, we get $P^* = (3,4,\dots,l,2,1)$ is a member of~$\mathcal{R}(P)$.
Furthermore, by taking $N^c(1;P^*) = A$, we have $S(P^*) = \{3,4\}$.
Apply Claim~\ref{clm:|C|}~(f) and~(c) with $P'' = P^*$, we deduce that $N^c(3;P^*) = \{4\} \cup \{t_3, t_3+1\}\cup \{l,2\}$ for some $t_3$ as $r(P^*) =  3$.
In particular, $4,l \in N^c(3,P^*)$ and so $c(3,4) \ne c(3,l)$.
However, recall the $s=3$, so $c(3,4)=c(3,l)$ by Lemma~\ref{lma:keyproperties}~(a).
This is a contradiction.

If $t=4$, then $A = \{2,4,5,l-1,l\}$ by~\eqref{eqn:clm:S>=3A3}.
By Claim~\ref{clm:|C|}~(e) with $b = 2$, we get $P^* = (3,4,\dots,l,2,1)$ is a member of~$\mathcal{R}(P)$.
Furthermore, by taking $N^c(1;P^*) = A $, we have $S(P^*) = \{4,5\}$.
Note that $s = 3$ and so $c(3,l) = c(3,4)$ implying that $l \notin N^c(3;P^*)$ (as $4 \in N^c(3;P^*)$ by definition).
Apply Claim~\ref{clm:|C|}~(f) and~(c) with $P'' = P^*$, we deduce that there exists an integer $t_3 \in [5,l-2]$ such that $N^c(3;P^*) = \{4\} \cup \{t_3, t_3+1\}\cup \{1,2\}$ as $r(P^*) =  4$.
Moreover, 
\begin{align}
c(3,4) \ne  c(3,t_3) = c(t_3, t_3+1) \ne c(t_3, t_3-1). \label{eqn:t3}
\end{align}
by Claim~\ref{clm:|C|}~(d) (taking $P = P^*$).
Next apply Claim~\ref{clm:|C|}~(e) to $P$ with $b = 3$, we get $P^{\star} = (4,5,\dots,l,3,2,1) \in \mathcal{R}(P)$.
Furthermore, by taking $N^c(1;P^{\star}) = A$, we have $S(P^{\star}) = \{4,5\}$.
Apply Claim~\ref{clm:|C|}~(f) with $P'' = P^{\star}$, we deduce that $N^c(4;P^{\star}) = \{5\} \cup \{t_4, t_4+1\}\cup \{2,3\}$ for some $t_4$ as $r(P^{\star}) =  4$.
Moreover, we have
\begin{align}
c(4,5) \ne c(4,2) = c(2,3) \ne c(1,2), \label{eqn:t4}
\end{align}
where the equality is due to~\eqref{eqn:l-1A} taking $P = P^{\star}$. 
Recall that $l-1 \in A = N^c(1;P)$ and $3 \in S$, so by Lemma~\ref{lma:keyproperties}~(a), we have 
\begin{align*}
c(1,2) \ne c(1,l-1) = c(l-1,l-2) \ne c(l-1,l) \ne c(3,l) =  c(3,4).
\end{align*}
Recall $t_3 \in [5,l-2]$.
Together with~\eqref{eqn:t3} and \eqref{eqn:t4}, we conclude that if $t_3 \ge 6$, then $(5,4,2,1,l-1,l,3,t_3,t_3 -1)$ is p.c..
Similarly if $t_3 = 5$, then $(4,2,1,l-1,l,3,5,4)$ is a p.c. cycle.
Therefore, $C' = (4,2,1,l-1,l,3,t_3,t_3 -1, \dots, 4)$ is a p.c. cycle satisfying~\eqref{eqn:C-keyproperty} with at least $6$ vertices.
So Lemma~\ref{lma:crossing} holds.

\noindent\textbf{Case 3: $w = l$.}
By~\eqref{eqn:clm:S>=3A}, $A = [2,r] \cup \{t,t+1\} \cup \{l\}$.
By Claim~\ref{clm:|C|}~(d), we have $c(1,t+1) = c(t+1,t+2) \ne c(t+1,t)$.
Hence $(1,2, \dots, t+1,1)$ is a p.c cycle satisfying~\eqref{eqn:C-keyproperty}.
This implies that $t+1 \le d$.
In fact, we have $t+1 = d$ as $|A| =d $ by Claim~\ref{clm:|C|}~(b).
Therefore, $A=[2,d] \cup \{l\}$.
In particular, $u = d$, $t=d-1$, $r = d-2$ and $s = d-1$, so  
\begin{align}
 \{d-2,d-1,l-1\} \subseteq  B \subseteq [d-2,l-1]\label{eqn:w=l1}
\end{align}
and $S=\{d-2,d-1\}$.
Let $b \in B \cap [d,l-2]$.
By the definition of~$B$, $c(l,b) \ne c(l,l-1) = c(1,l)$.
If $c(l,b) \ne c(b-1,b)$, then $(1,2,\dots,b,l,1)$ is a p.c. cycle of length~$b+1 \ge d+1$ satisfying~\eqref{eqn:C-keyproperty}, a contradiction.
Thus, 
\begin{align}
c(l,b) & = c(b-1,b) &\text{for all } b \in B \cap [d,l-2]. \label{eqn:w=l}
\end{align}
If $b =d$, then $c(l,d-1) = c(d-1,d) = c(l,d)$ as $d-1 \in S$, which contradicts the fact that $d-1,d \in B = N^c(l)$.
If $b = d+1$, then $(1,d,d+1,l)$ is a monochromatic path of length~3.
Let $G'$ be the edge-coloured subgraph of $G$ obtained by removing the edge $(d,d+1)$.
Note that $\delta^c(G')=d$.
The p.c. path $$P''=(d+1,d+2, \dots, l , d-1, d-2, \dots, 1 ,d)$$ can be obtained by a rotation of~$P$ with pivot point~$d-1$ and endpoint~$1$ followed by a rotation with pivot point~$d$ and endpoint~$d$.
Hence, $P'' \in \mathcal{R}(P)$ is maximal in~$G$ and also in~$G'$ contradicting the edge-minimality of~$G$.
Therefore, $B \cap \{d,d+1\}=\emptyset$, so \eqref{eqn:w=l1} becomes 
\begin{align}
B \subseteq \{d-2,d-1\} \cup [d+2,l-1]. \label{eqn:w=l2}
\end{align}
Recall that $l \in A$, so the p.c. path $P_0 = (l-1, l-2,\dots, 1,l) \in \mathcal{R}(P)$.
Set $N^c(l;P_0) = B \cup \{2\} \setminus \{l-1\}$.
By~\eqref{eqn:w=l}, $B \cap [d,l-2] \subseteq S(P_0)$.
Therefore, Claim~\ref{clm:t-1notinS} and \eqref{eqn:w=l2} imply that 
\begin{align*}
5 \le & d \le |B| =  |\{d-2,d-1, l-1\} \cup (B \cap [d+2,l-2])| \\
\le  & 3+|S(P_0)| \le 3+ |S(P)| \le 5,
\end{align*}
so $d = 5$ and $r=3$.
If we replace $P$ with $P_0$ and repeat all the arguments in the proof of this claim, then we can deduce that $B = \{3,4,l-4,l-3,l-1 \}$.
By~\eqref{eqn:w=l2}, $7 = d+2 \le l-4$ and so $l\ge 11$, which implies Lemma~\ref{lma:crossing}.
This completes the proof of the claim.
\end{proof}

By Claim~\ref{clm:t-1notinS}, we may assume that $t \ge r+3$.
Fix $N^c(r+1;P)$ and $N^c(r+3;P)$.
Next, we are going to show that $(r+1,r+3)$ is an edge such that
\begin{align}
c(r+3,r+4) \ne c(r+1,r+3) \ne c(r+1,r+2). \label{eqn:c(r+1,r+3)}
\end{align}
First, apply Claim~\ref{clm:|C|}~(e) with $ b=r$ and obtain 
\begin{align*}
P_1 = (r+1,r+2 \dots, l , r,\dots,1 ) = (x_1, x_2,\dots, x_l) \in \mathcal{R}(P).
\end{align*}
Set $N^c(1;P_1) = A$ if $r \ne 1$ and $N^c(1;P_1) = (A \cup \{l\}) \setminus \{2\}$ if $r = 1$.
By Claim~\ref{clm:|C|}~(e), $S(P_1) = [t,u]$.
Since $t \ge r+3$, $r(P_1) \ne r+1$. 
Set $N^c(r+1;P_1)= N^c(r+1;P)$, so $r,r+2 \in N^c(r+1;P_1)$.
By Claim~\ref{clm:|C|}~(f) and (c) taking $P'' = P_1$, we have 
\begin{align}
N^c(r+1;P_1) = \{ x_j : j\in [2,r_1] \cup [t_1,u_1] \cup [w_1,l] \} \nonumber
\end{align}
for some $2 \le r_1 < t_1 < u_1 < w_1 \le l$.
Note that $x_{r_1} = r(P_1) = t \ge r+3$, so $r+3 \in N^c(r+1;P_1) = N^c(r+1;P)$.
Hence, $\{r,r+2,r+3\} \subseteq N^c(r+1;P)$.
Second, note that $r+2 \in S$ by Claim~\ref{clm:|C|}~(b) and Claim~\ref{clm:S>=3}.
Apply Claim~\ref{clm:|C|}~(e) with $ b=r+2$ to $P$ and obtain 
\begin{align*}
P_2 = (r+3, r+4,\dots, l , r+2,r+1, \dots,1 ) = (y_1, y_2, \dots, y_l) \in \mathcal{R}(P).
\end{align*}
Set $N^c(1;P_2)  = A$ and $N^c(r+3,P_2) =N^c(r+3,P)$.
Since $r+2 < t$, by Claim~\ref{clm:|C|}~(e) we have $S(P_2) = [t,u]$.
By Claim~\ref{clm:|C|}~(f) and (c) taking $P'' = P_2$, we have $N^c(r+3;P_2)  =  \{y_j : j \in A_2\}$, where
\begin{align}
A_2  = \left\{ \begin{array}{ll}
	[2,r_2] \cup [t_2,u_2] \cup [w_2,l]	& \text{if $t \ne r+3$,}\\
	\{ 2 \} \cup [t_2,u_2] \cup [w_2,l-1]	& \text{if $t = r+3$,}
    \end{array} \right. \label{eqn:A2}
\end{align}
for some $2 \le r_2 < t_2 < u_2 < w_2 \le l$, where $y_{r_2} = t$.
Claim~\ref{clm:|C|}~(d) implies that 
\begin{align}
c(r+3, y_j) & = c(y_j, y_{j+1}) &\text{ for $t_2 \le j \le u_2$.} \label{eqn:r_3}
\end{align}
Let $y_{j'} = r+2$, so $y_{j'-1} = l$.
Recall that $r+2 \in S$, so 
\begin{align}
c(r+3,y_{j'}) = c(r+3,r+2) = c(l,r+2) = c(y_{j'-1},y_{j'})
\nonumber
\end{align}
by Lemma~\ref{lma:keyproperties}~(a).
Since $r+2 \in N^c(r+3;P) = N^c(r+3;P_2)$, \eqref{eqn:A2} and \eqref{eqn:r_3} imply that $w_2 \le j' <l$.
Therefore, $[2,r+2] \subseteq  N^c(r+3;P_2)$.
In particular, $r+1 \in N^c(r+3;P_2) = N^c(r+3,P)$.
In summary, we have shown that $r+3 \in N^c(r+1;P)$ and $r+1\in N^c(r+3;P)$, so \eqref{eqn:c(r+1,r+3)} holds.

Recall Claim~\ref{clm:|C|}~(a) and (b) that $|C_0|=|[1,s]|+|[w,l]| = d$.
If $s+1 \in B$, then $C' = (1,2,\dots,s+1,l,l-1,\dots,w,1)$ is a p.c. cycle of length~$d+1$ satisfying~\eqref{eqn:C-keyproperty}, because $c(l,s+1) \ne c(l,s) = c(s,s+1)$ by Lemma~\ref{lma:keyproperties}~(a).
Hence, $s+1 \notin B$.
Claim~\ref{clm:S>=3} implies that $r+2 \in S$ and so $c(l,l-1) \ne c(l,r+2) = c(r+2,r+3) \ne c(r+2,r+1)$. 
Together with~\eqref{eqn:c(r+1,r+3)}, $C''=(r+1,r+2,l,l-1,\dots,r+3,r+1)$, see Figure~\ref{fig:C3}, is a p.c. cycle containing $( B \cup \{ l,s+1\} )\setminus \{r\}$.
Moreover, $C''$ satisfies~\eqref{eqn:C-keyproperty}, a contradiction as $|C''| \ge d+1$.
The proof of the Lemma~\ref{lma:crossing} is completed.
\begin{figure}[tbp]
\begin{center}
\includegraphics[scale=0.6]{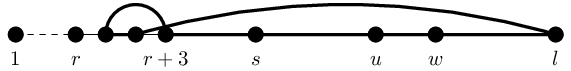}
\end{center}
\caption{Cycle $(r+1,r+2,l,l-1\dots,r+3,r+1)$}
\label{fig:C3}
\end{figure}
\end{proof}

\section{Graphs with short p.c. cycles} \label{sec:k<d}

This section concerns graphs such that no p.c. cycle has length more than some fixed~$k$.
First we prove Proposition~\ref{prp:upper}.

\begin{proof}[Proof of Proposition~\ref{prp:upper}]
Fix $k$ and we proceed by induction on $d$.
For $d=k-1$, Proposition~\ref{prp:upper} holds by considering $\widetilde{G}(d;p)$ for $p \ge d$ as defined in Example~\ref{exm:1}.
Thus, we may assume $d >k-1$.
Let $\mathcal{G}(d-1,k)$ be the family of edge-coloured graphs $G$ with $\delta^{c}(G) \ge d-1$ such that the longest p.c. paths and p.c. cycles in $G$ are of lengths $k 2^{d-k+1}-2$ and~$k-1$ respectively.
Note that $\mathcal{G}(d-1,k)$ exists by induction hypothesis.
Next, we take $p \ge d$ vertex-disjoint copies of members of $\mathcal{G}(d-1,k)$, $H_1, \dots, H_{p}$.
Take a new vertex $x$ and add an edge of a new colour $c_j$ between~$x$ and every vertex of~$H_j$ for each $j \in [p]$.
Call the resulting graph $G'$.
It is easy to see by induction on~$d$ that every vertex in $G'$ has minimum colour degree at least~$d$.
Moreover, the longest p.c. path and p.c. cycle in $G'$ have lengths $k 2^{d-k+2}-2$ and~$k-1$ respectively.
\end{proof}

We are going to prove Theorem~\ref{thm:k<d} in the remainder of this section.
First, we will need the following definitions.
Let $c$ be an edge-colouring of a graph~$G$ such that $\delta^c(G) =d \ge 3$.
Let $P=(1,2, \dots, l)$ be a p.c. path in $G$.
Define $f_i(P)$ to be the resultant path after a rotation of $P$ pivoting at the $i$th element with the last vertex as the fixed endpoint.
Similarly, define $g_j(P)$ to be the resultant path after a rotation of $P$ pivoting at the $j$th element with the first vertex as the fixed endpoint.
Since $P$ is considered as an $l$-tuple, we consider $f_i$ and $g_j$ as permutations on~$P$.
For example, $f_3  \circ g_1(1,2,3,4,5,6) = f_3 (1,6,5,4,3,2) = (6,1,5,4,3,2)$.
Furthermore, we only consider $f_i(P)$ and $g_j(P)$ if $f_i(P)$ and $g_j(P)$ are p.c. paths respectively.
This means that if $1 < i <l$ and $c(1,2) \ne c(1, i) \ne c(i, i+1)$, then $f_i(P)$ is defined, and a similar statement for $g_j(P)$.
Let $\mathcal{R}'(P)$ be the set of p.c. paths that can be obtained by a sequence of rotations of~$P$.
Note that $\mathcal{R}(P) = \{P', h(P') : P' \in \mathcal{R}'(P)\}$, where $h(P')$ is the reflection of $P'$.
We study some basic properties of $\mathcal{R}'(P)$ in the coming proposition.

\begin{prp} \label{prp:figj}
Let $c$ be an edge-colouring of a graph~$G$ such that $\delta^c(G) =d \ge 3$.
Let $P$ be a properly coloured path in~$G$ with $V(P) = V(P')$.
Then the following statements hold:
\begin{description}
\item[(a)] If $P'$ is a p.c. path and $f_i(P') \in \mathcal{R}'(P)$, then $P' \in \mathcal{R}'(P)$.
\item[(b)] If $P'$ is a p.c. path and $g_j(P') \in \mathcal{R}'(P)$, then $P' \in \mathcal{R}'(P)$.
\item[(c)] For $i \le j$, $f_i$ and $g_j$ commute.
\end{description}
Furthermore, suppose that $P'$ has no crossing for all $P' \in \mathcal{R}'(P)$ and all choices of $N^c(x;P')$ and $N^c(y;P')$, where $x$ and $y$ are the endpoints of $P'$.
Then
\begin{description}
\item[(d)] every $P' \in \mathcal{R}'(P)$ can be obtained from $P$ by a sequence of $f_{i_1},\dots, f_{i_a}$ followed by a sequence of $g_{j_1},\dots, g_{j_b}$ and visa versa.
\item[(e)] there exist integers $i_0 \le j_0$ depending only on $\mathcal{R}'(P)$ such that $i_{a'} \le i_0$ and $ j_0 \le j_{b'}$ for $a' \in [a]$ and $b' \in [b]$.
\end{description}
\end{prp}

\begin{proof}
Let $P'=(1, 2 \dots, l)$ be a p.c. path and so $f_i(P') = (i-1, i-2, \dots,1,i,i+1, \dots,l)$.
Since $P'$ is a p.c. path, we must have $c(i-1,i-2) \ne c(i-1,i)$.
If $i <l$, then $c(i-1,i)\ne c(i,i+1)$.
Thus, $P' = f_i \circ f_i(P') \in \mathcal{R}'(P)$ and so (a) holds.
By a similar argument, (b) holds.
Note that $f_i$ (and $g_j$) reverses the ordering in the first $(i-1)$ elements (and the last $(l-j)$ elements respectively).
Hence, (c) follows easily.

Assume that $P'$ has no crossing for all $P' \in \mathcal{R}'(P)$ and any choices of colour neighbourhoods.
In order to prove (d), it is enough to show that if $f_i \circ g_j(P') \in \mathcal{R}'(P)$, then $f_i \circ g_j(P') = g_j \circ f_i(P')$.
Suppose that $f_i \circ g_j(P') \in \mathcal{R}'(P)$.
By (a) and~(b), we have $g_j(P'),P' \in \mathcal{R}'(P)$.
Recall that $P' = (1,2, \dots, l)$, so $g_j(P') = (1, 2, \dots, j, l, l-1, \dots, j+1)$.
Note that $i \in N^c(1;g_j(P'))$ for some $N^c(1;g_j(P'))$ as $f_i ( g_j(P'))$ is defined.
Fix one such $N^c(1;g_j(P'))$.
Since $P'$ is p.c., we may pick $N^c(j+1;g_j(P'))$ such that $j \in N^c(j+1;g_j(P'))$.
Recall that $g_j(P')$ has no crossing, so 
\begin{align}
i \le \max\{i' \in N^c(1;g_j(P')) \} \le \min\{ j' \in N^c(j+1;g_j(P'))\} \le j. \label{eqn:prp:figj}
\end{align}
By~(c), $f_i \circ g_j(P') = g_j \circ f_i(P')$. 
Hence, (d) holds.
Let $i_0$ (and $j_0$) be the maximal integer $i$ (and the minimal integer $j$) such that $x_{i_0} \in N^c(x_1 ; P'')$ (and $x_{j_0} \in N^c(x_l ; P'')$) for some $P''= (x_1, \dots, x_l) \in \mathcal{R}'(P)$.
Moreover, (e) follows from (d) and~\eqref{eqn:prp:figj}.
\end{proof}

Let $G$ be an edge-colouring graph such that no p.c. cycle has length more than some fixed~$k$.
The next lemma show that the length of every maximal p.c. path grows exponentially in~$\delta^c(G)$. 
Thus, Lemma~\ref{lma:k<dstronger} trivially implies Theorem~\ref{thm:k<d}.
The main idea of the proof of the lemma is as follows.
Let $P=(1,2, \dots, l)$ be a maximal p.c. path in~$G$.
By Lemma~\ref{lma:crossing2} and Lemma~\ref{lma:crossing}, $P$ does not have a crossing.
Our aim is to find integers $1 < x < y <l$ such that $P_x = (1, 2, \dots, x)$ and $P_y = (y+1, y+2, \dots, l)$ are maximal p.c. path in $G_x = G \setminus \{x+1\}$ and $G_y = G \setminus \{y-1\}$ respectively. 
Clearly, $\delta^c(G_x) \ge \delta^c(G)-1$.
By inducting on $\delta^c(G)$, we can deduce that $P_x$ is very long (exponentially in~$\delta^c(G_x)$), and a similar statement holds for $P_y$.
Thus, $P$ is also very long.

\begin{lma} \label{lma:k<dstronger}
Let $k \ge 3$ and $d \ge \lceil 3k/2 \rceil-3$ be integers.
Let $c$ be an edge-colouring of a graph~$G$ such that $\delta^c(G) =d$.
Suppose $G$ does not contain any properly coloured cycle of length at least~$k$.
Then, every maximal properly coloured path in $G$ has length at least $k2^{d-\lceil 3k/2 \rceil+4}-2$.
\end{lma}

\begin{proof}
Let $P=(1,2,\dots,l)$ be a maximal p.c. path in~$G$.
We are going to show that $l \ge k2^{d-\lceil 3k/2 \rceil+4}-1$ by induction on~$d$.
If $d = \lceil 3k/2 \rceil -3\ge k-1$, then no $P' \in \mathcal{R}'(P)$ has a crossing for all choices of colour neighbourhoods.
Otherwise Lemma~\ref{lma:crossing2} and Lemma~\ref{lma:crossing} imply that $G$ contains a p.c. cycle of length at least $d+1 \ge k$ or $|P| = l \ge 2d+1 \ge 2k-1$ as required.
Since $P$ does not have a crossing, we have
\begin{align*}
 l \ge |(N^c(1;P) \cup \{1\}) \cup (N^c(l;P) \cup \{l\})| \ge 2(d+1)-1 \ge 2k-1.
\end{align*}
Thus, the lemma is true for $d = \lceil 3k/2 \rceil -3$.
Hence, we may assume that $d \ge \lceil 3k/2 \rceil -2 \ge k$.

Define $X = X(P)$ to be the set of all possible $i_1$ such that there exists a path $P' = (i_1, \dots, i_l) \in \mathcal{R}'(P)$.
Similarly, define $Y = Y(P)$ to be the set of all possible $i_l$ such that $P' = (i_1, \dots, i_l) \in \mathcal{R}'(P)$.
Clearly, $1 \in X$ and $l \in Y$.
Let $x = \max \{i \in X\}$ and $y = \min \{j \in Y\}$.
If $f_i(P') \in \mathcal{R}'(P)$, then $i \le x+1$.
If $g_j(P') \in \mathcal{R}'(P)$, then $j \ge y-1$.
By Proposition~\ref{prp:figj}~(e), $x+1 \le y-1$.
Since $P$ is maximal, $N^c(1;P) \subseteq [l]$ for all choices of~$N^c(1;P)$.
If $i' \in N^c(1;P)$ is maximal, then $c(1,i') = c(i',i'-1)$ or else $(1,2, \dots,i',1)$ is a cycle of length at least $d+1$.
Hence, $i' -1 \in X$ as $f_{i'}(P) \in \mathcal{R}'(P)$.
Thus, we have
\begin{align}
x & = \max \{i-1 : i \in N^c(i_1;P') \text{ for all }P' = (i_1, \dots, i_l) \in \mathcal{R}'(P) \text{ and all $N^c(i_1;P')$} \}  \ge d \ge k. \label{eqn:lma:k<dstronger:x}
\end{align}
Without loss of generality, we may assume that $x+1 \in N^c(1;P)$.
By a similar argument, if $j' \in N^c(l;P)$ is minimal, then $j' +1 \in Y$.
By Proposition~\ref{prp:figj}~(d) and (e), we may further assume that $y-1 \in N^c(l;P)$.

Let $P_x= (1,2, \dots ,x)$ and $P_y = (y, y+1,\dots, l)$ be p.c. paths.
Let $G_x = G \setminus \{x+1\}$ and $G_y = G \setminus \{y-1\}$.
Clearly, $\delta^c(G_x) , \delta^c(G_y) \ge d-1$.
Suppose that $P_x$ and $P_y$ are maximal in $G_x$ and $G_y$ respectively.
By the induction hypothesis, $|P_x|, |P_y| \ge k2^{d-\lceil 3k/2 \rceil+3}-1$.
Note that $x+1$ is not a vertex in $P_x$ nor $P_y$, so
\begin{align*}
l \ge |P_x| + |P_y| +1 \ge 2(k2^{d-\lceil 3k/2 \rceil+3}-1)+1 = k2^{d-\lceil 3k/2 \rceil+4}-1
\end{align*}
as required.
Hence, in proving the lemma, it suffices (by symmetry) to show that $P_x$ is maximal in~$G_x$.

\begin{clm} \label{clm:lma:k<dstronger:1}
Let $P' =(i'_1,i'_2, \dots, i'_x)$ and $P'' = (i''_1,i''_2, \dots, i''_x)$ be p.c. paths with $V(P') = V(P'') = [x]$.
\begin{description}
\item[(a)] If $P' = f_{j_a} \circ \dots \circ f_{j_1} ( P'' )$, then $i'_x = i''_x$ and $c(i'_x,i'_{x-1}) = c(i''_x, i''_{x-1})$.
Moreover, if $P'' = P_x$, then $(i'_1,i'_2, \dots, i'_x,x+1,x+2, \dots, l) \in \mathcal{R}'(P)$.
\item[(b)] If $P' = g_{j_b} \circ \dots \circ g_{j_1} (P'')$, then $i'_1 = i''_1$ and $c(i'_1,i'_{2}) = c(i''_1,i''_2)$.
Moreover, if $P'' = P_x$, then $(i'_x,i'_{x-1}, \dots, i'_1,x+1,x+2, \dots, l) \in \mathcal{R}'(P)$.
\end{description}
\end{clm}

\begin{proof}
Note that $f_j$ fixes the last two elements unless $j = x$.
Thus, in proving the first assertion of (a) it is enough to consider the case when $a=1$ and $j_1 =x$.
Note that $i'_x = i''_x$ and $i''_{x-1} = i'_1$.
Also $c(i'_1,i'_2) = c(i''_{x-1} ,i''_{x-2}) \ne c(i''_{x-1} ,i''_{x}) =  c(i'_1, i'_{x})$ as $P''$ is a p.c. path.
If $c(i'_x,i'_{x-1}) \ne c(i''_x, i''_{x-1}) = c(i'_x, i'_{1})$, then $(i'_1,i'_2, \dots, i'_x, i'_1)$ is a p.c. cycle of length $x \ge k$ by~\eqref{eqn:lma:k<dstronger:x}, a contradiction.
Thus, we have $c(i'_x,i'_{x-1}) = c(i''_x, i''_{x-1})$.
Now suppose that $P'' = P_x$.
Recall that $c(x-1,x) \ne c(x,x+1)$ as $P$ is a p.c. path, so the `moreover' statement follows.

By a similar argument, first assertion of (b) also holds.
Since $x+1 \in N^c(1;P)$, we get $c(1,2) \ne c(1,x+1)$.
Also, $c(1,x+1) = c(x,x+1)$ or else $(1, 2, \dots, x+1, 1)$ is a p.c. cycle of length $x >k$ by~\eqref{eqn:lma:k<dstronger:x}.
Thus, the `moreover' statement of (b) follows.
\end{proof}

First suppose that no $P'_x \in \mathcal{R}'(P_x)$ has a crossing for all choices of colour neighbourhoods in $G_x$.
Let $X_x = X(P_x)$ and $Y_x = Y(P_x)$ with respect to~$G_x$.
If $P'_x = (i_1, \dots, i_x) \in \mathcal{R}'(P_x)$ is extensible in $G_x$, then there exists $u \notin [x+1]$ such that $(i_1, \dots, i_x, u)$ or $(u,i_1, \dots, i_x)$ is a p.c. path in $G_x$.
Assume that $(i_1, \dots, i_x, u)$ is a p.c. path in $G_x$.
Note that $c(i_x, u ) \ne c(i_x, i_{x-1})$.
By Proposition~\ref{prp:figj}~(d),
$P'_x = f_{i_a} \circ \dots \circ f_{i_1} \circ g_{j_b} \circ \dots \circ g_{j_1} (P_x)$.
By Proposition~\ref{prp:figj}~(a), we may assume without loss of generality that 
$P'_x = g_{j_b} \circ \dots \circ g_{j_1} (P_x)$ as the statement $c(i_x, u ) \ne c(i_x, i_{x-1})$ still holds by Claim~\ref{clm:lma:k<dstronger:1}~(a).
Fix $N^c_{G_x}(i_x;P_x')$ such that $u \in N^c_{G_x}(i_x;P_x')$.
By Claim~\ref{clm:lma:k<dstronger:1}~(b), we have $P' = (i_x, i_{x-1}, \dots, i_1, x+1,x+2, \dots,l) \in \mathcal{R}'(P)$.
Pick $N^c_G(i_x;P')$ such that $N^c_{G_x}(i_x;P_x') \subseteq N^c_G(i_x;P')$.
Hence, $u \in N^c_G(i_x;P')$ and $u \notin [x+1]$.
If $u \in [l]$, then $u > x+1$ contradicting~\eqref{eqn:lma:k<dstronger:x}.
If $u \notin[l]$, then $P'$ is extensible contradicting the maximality of~$P$.
A similar argument also holds if $(u,i_1, \dots, i_x)$ is a p.c. path in $G_x$.
Thus, every $P'_x \in \mathcal{R}'(P)$ is not extensible in $G_x$ and so $P_x$ is maximal in~$G_x$ as required.

Now suppose there exists a p.c. path $P^* = (i_1, \dots, i_x)  \in \mathcal{R}'(P_x)$ that has a crossing in $G_x$ for some $N^c_{G_x}(i_1;P')$ and $N^c_{G_x}(i_x;P^*)$.
The next claim allows us to assume without loss of generality that $i_1, i_x \in X$.

\begin{clm}	\label{clm:lma:k<dstronger:2}
There exists a p.c. path $P^* = (i_1, \dots, i_x)  \in \mathcal{R}'(P_x)$ such that the following statements hold:
\begin{itemize}
	\item[\rm (a)] there exist $N^c_{G_x}(i_1;P^*)$ and $N^c_{G_x}(i_x;P^*)$ such that
$P^*$ has a crossing.
	\item[\rm (b)] $i_1,i_x \in X$.
	\item[\rm (c)] $P^*$ is not extensible in $G_x$.
Moreover, $N^c_G(i_1;P^*) \cup N^c_G(i_x;P^*) \subseteq [x+1]$ for any choices of colour neighbourhoods.
	\item[\rm (d)] If $x+1 \in N^c_G(i_1,P^*)$, then $(i_x,i_{x-1}, \dots, i_1,x+1, x+2, \dots,  l) \in \mathcal{R}'(P)$.
	\item[\rm (e)] If $x+1 \in N^c_G(i_x,P^*)$, then $(i_1,i_2, \dots, i_x,x+1, x+2,\dots,  l) \in \mathcal{R}'(P)$.
\end{itemize}
\end{clm}

\begin{proof}[Proof of claim]
Choose a p.c. path $P^* = (i_1, \dots, i_x)  \in \mathcal{R}'(P_x)$, which has a crossing for some $N^c_{G_x}(i_1;P^*)$ and $N^c_{G_x}(i_x;P^*)$.
Notice that such $P^*$ exists and satisfies~(a).
For integers $a' \ge 0$, let $\mathcal{R}'_{a'}(P)$ be the set of $P' \in \mathcal{R}'(P_x)$ that can be obtained from $P_x$ by using precisely $a'$ rotations.
Clearly, $\mathcal{R}'_0(P_x) = \{P_x\}$.
We further assume that $P^*$ is chosen such that $P^* \in \mathcal{R}'_{a_0}(P_x)$ with $a_0$ is minimal.
This implies that every $P' \in \bigcup_{0 \le a' < a_0} \mathcal{R}'_j(P_x)$ has no crossing for all choices of colour neighbourhoods.

Next we are going to show that, for all $ 0 \le a' \le a_0$, every $P_{a'} \in \mathcal{R}'_{a'}(P_x)$ can be written as 
\begin{align}
 P_{a'} = g_{j'_b} \circ \dots \circ g_{j'_1} \circ f_{j_a} \circ \dots \circ f_{j_1} (P_x)  = f_{j_a} \circ \dots \circ f_{j_1} \circ g_{j'_b} \circ \dots \circ g_{j'_1}(P_x) \label{eqn:P'rot3}
\end{align}
for some $j_1, \dots, j_a, j'_1, \dots, j'_b$ with $a +b = a'$.
We now prove~\eqref{eqn:P'rot3} by induction on~$a'$.
Note that \eqref{eqn:P'rot3} holds trivially for $a' \le 1$ and so are going to show that \eqref{eqn:P'rot3} holds for $a' \ge 2$.
Let $P_{a'} \in \mathcal{R}'_{a'}(P_x)$, so $P_{a'} = f_{j}(P')$ or $P_{a'} = g_{j}(P')$ for some $j$ and $P'\in \mathcal{R}'_{{a'}-1}(P_x)$.
We will only consider the case when $P_{a'} = f_{j}(P')$ (as similar argument holds for the other cases).
By induction hypothesis, we can write $P' = g_{j'_b} \circ \dots \circ g_{j'_1} \circ f_{j_a} \circ \dots \circ f_{j_1} (P_x)$ with $a+b = a'-1$.
In order to prove \eqref{eqn:P'rot3} holds for~$P_{a'}$, by Proposition~\ref{prp:figj}~(c), it is enough to show that $\max \{j'_1, \dots, j_b'\} \le j$.
Moreover, it suffices to consider the case when $b = 1$ and $j' = j'_1$.
Let $P' = (i_1', i_2', \dots, i_x')$.
Let $P'' = f_{j_a} \circ \dots \circ f_{j_1} (P_x)$, so $P'' \in \mathcal{R}'_{a'-2}(P_x)$ and $P' = g_{j'}(P'')$.
This means that $P'' = g_{j'}(P')$ and so we can pick $N^c_{G_x}(i'_x; P')$ with $i'_{j'} \in N^c_{G_x}(i'_x; P')$.
Since $P_{a'} = f_{j}(P')$, we can pick $N^c_{G_x}(i'_1; P')$ with $i'_j \in N^c_{G_x}(i'_x; P')$.
Recall that $P' \in \mathcal{R}'_{a'-1}(P_x)$ has no crossing, so $j'\le j$.
Hence, \eqref{eqn:P'rot3} holds.

Recall that $P^* \in \mathcal{R}_{a_0}'$, so we may assume that 
\begin{align}
 P^* = g_{j'_b} \circ \dots \circ g_{j'_1} \circ f_{j_a} \circ \dots \circ f_{j_1} (P_x) = \mathcal{G} \circ \mathcal{F} (P_x) = \mathcal{F} \circ \mathcal{G} (P_x), \label{eqn:P'rot}
\end{align}
where $a +b = a_0$, $\mathcal{G} = g_{j'_b} \circ \dots \circ g_{j'_1}$ and $\mathcal{F} = f_{j_a} \circ \dots \circ f_{j_1} $.
Let 
\begin{align}
P' = \mathcal{F} (P_x) = (i'_1,i'_2, \dots, i'_x).	\label{eqn:P'rot2}
\end{align}
By Claim~\ref{clm:lma:k<dstronger:1}~(a), we have $i'_1 \in X$.
Since $P^* = \mathcal{G}(P') = g_{j'_b} \circ \dots \circ g_{j'_1} (P')$ and each $g_{j'_{i'}}$ fixed the first vertex, we deduce that $i'_1 = i_1$ and so $i_1 \in X$.
To show that $i_x \in X$, we consider $P^* =\mathcal{F} \circ \mathcal{G}(P_x)$ instead.
Let 
\begin{align}
P'' = \mathcal{G}(P_x) = (i''_1,i''_2, \dots, i''_x).	\nonumber
\end{align}
By Claim~\ref{clm:lma:k<dstronger:1}~(b), we have $i''_x \in X$.
Since $P^* = \mathcal{F}(P') = f_{j_a} \circ \dots \circ f_{j_1} (P'')$ and each $f_{j_{i}}$ fixed the last vertex, we deduce that $i''_x = i_x$ and so $i_x \in X$.
Hence, $P^*$ satisfies~(b).

Suppose that $N^c_{G_x}(i_1;P^*) \not\subseteq [x]$.
Let $u \in N^c_{G_x}(i_1;P^*) \setminus [x]$, so $c(i_1,i_2) \ne c(i_1, u)$.
Recall \eqref{eqn:P'rot2} that $P' = \mathcal{F} (P_x) = (i'_1,i'_2, \dots, i'_x)$ and $P^* = \mathcal{G} (P')$.
By Claim~\ref{clm:lma:k<dstronger:1}~(b), we have $c(i'_1,i'_2) = c(i_1,i_2)$ and $i_1 = i'_1$.
So we can set $N^c_{G_x}(i'_1;P') = (N^c_{G_X}(i_1;P^*) \setminus \{i_2 \}) \cup \{i'_2\}$.
By Claim~\ref{clm:lma:k<dstronger:1}~(a), we have $$\widetilde{P} = (i'_1,i'_2, \dots, i'_x, x+1, x+2, \dots, l) \in \mathcal{R}'(P).$$
Pick $N^c_{G}(i'_1;\widetilde{P})$ such that $N^c_{G_x}(i'_1;P') \subseteq N^c_{G}(i'_1;\widetilde{P})$.
Since $u \in N^c_{G_x}(i'_1;\widetilde{P})$ and $u \notin [x]$, we have $c(i'_1,i'_2) \ne c(i_1',u)$ and $u \notin [x+1]$.
If $u \in [l]$, then $u > x+1$ contradicting~\eqref{eqn:lma:k<dstronger:x}.
If $u \notin[l]$, then $\widetilde{P}$ is extensible contradicting the maximality of~$P$.
Therefore, $N^c_{G_x}(i_1;P^*) \subseteq [x]$ for all choices of colour neighbourhoods.
This implies that $N^c_{G}(i_1;P^*) \subseteq [x+1]$ for all choices of colour neighbourhoods.
By a similar argument, we have $N^c_{G_x}(i_x;P^*) \subseteq [x]$ and  $N^c_G(i_x;P^*) \subseteq [x+1]$.
So $P^*$ satisfies property~(c).

Now suppose that $x+1 \in N^c_G(i_1, P^*)$ and so $x+1 \in N^c_G(i'_1, \widetilde{P})$ by setting $N^c_G(i'_1, \widetilde{P}) = (N^c_{G}(i_1;P^*) \setminus \{i_2 \}) \cup \{i'_2\}$.
If $c(x+1, i'_x) \ne c(i'_1, x+1)$, then $(i'_1, \dots, i'_x, x+1)$ is a p.c. cycle of length at least $d+1 \ge k$, a contradiction.
Hence, $f_{x+1} ( \widetilde{P})$ exists.
Note that $f_{x+1}$ reserves the ordering of the first $x$ elements in $\widetilde{P}$, so we can view it as a reflection on $P'$.
Therefore,
\begin{align*}
 (i_x,i_{x-1}, \dots, i_1,x+1, x+2, \dots, l) =  f_{x-j'_b+1} \circ \dots \circ f_{x-j'_1+1} \circ f_x (\widetilde{P})
\end{align*}
is a member of $\mathcal{R}'(P)$, so (d) holds.
Finally, (e) is proved by a similar argument used to prove~(d).
\end{proof}

For convenience, we abuse the notation and assume that $P^* = (1,2, \dots,x)$, so $1$ and $x$ are not necessarily adjacent to $x+1$.
Let $N^c_{G_x}(1;P^*)$ and $N^c_{G_x}(x;P^*)$ such that $P^*$ has a crossing.
Note that \begin{align}
\text{$1 \notin N^c_{G_x}(x;P^*)$ or $x \notin N^c_{G_x}(1;P^*)$.} \label{eqn:lA1B:2}
\end{align}
Otherwise, $(1, 2, \dots, x,1)$ is a p.c. cycle of length $x \ge k$ by~\eqref{eqn:lma:k<dstronger:x}.
We may assume that $1 \notin N^c_{G_x}(x;P^*)$.
Let $r = r(P^*) = \min\{b \in B\}$, so $r \ge 2$.
Recall that $\delta^c(G_x) \ge d-1$ and $P^*$ is not extensible in $G_x$ by Claim~\ref{clm:lma:k<dstronger:2}~(c).
By Lemma~\ref{lma:keyproperties} taking $G = G_x$ and $P = P^*$, we can find an integer $s= s(P^*)$ satisfying Lemma~\ref{lma:keyproperties}~$(a)-(c)$.
Let $S = S(P^*) = [s] \cap N^c_{G_x}(x;P^*)$.
If $b \in N^c_{G_x}(x;P^*)$ and $b \le x-k+1$, then $c(x,b) = c(b,b+1)$ or else $(b,\dots,x,b)$ is a p.c. cycle of length at least~$k$.
Thus, 
\begin{align}
|S|\ge |N^c(x;P^*) \cap [x-k+1]| \ge d^c_{G_x}(x)-k+2. \label{eqn:lma:k<dstronger:|S|}
\end{align}
Let $P^{\star} = (x, x-1, \dots, 1)$ be the reflection of $P^*$.
Set $N^c_{G_x}(1;P^*)  =  N^c_{G_x}(1;P^{\star})$ and $N^c_{G_x}(x;P^*)  =  N^c_{G_x}(x;P^{\star})$.
So $P^{\star}$ has a crossing. 
Apply Lemma~\ref{lma:keyproperties} to obtain $s^{\star} = s(P^{\star})$ and set $S^{\star} = S(P^{\star}) = [ s^{\star} , x] \cap N^c_{G_x}(1;P^*)$.
By a similar argument, we have 
\begin{align}
|S^{\star}|\ge d^c_{G_x}(1)-k+2. \label{eqn:lma:k<dstronger:|S'|}
\end{align}
Recall that $r \ge 2$, so $s \ge 2$.
Hence, we can find $u=u(P^*),w=w(P^*) \in N^c_{G_x}(1;P^*)$ satisfying Lemma~\ref{lma:keyproperties}~$(d)-(f)$.
Since $c(1,s^{\star}) = c(s^{\star}, s^{\star}-1)$ by Lemma~\ref{lma:keyproperties}~(a), $u < s(P^{\star})$ and so $w \le s(P^{\star})$ by Lemma~\ref{lma:keyproperties}~(e).
This means that $S(P^{\star}) \subseteq [w,x]$.
By Lemma~\ref{lma:keyproperties}~(a) and~(f) and Lemma~\ref{lma:simplecycle}, 
\begin{align*}
C_0=(1,2\dots,s,x, x-1,\dots, w, 1)
\end{align*}
is a p.c. cycle.
By our assumption in the hypothesis of Lemma~\ref{lma:k<dstronger}, we know that $|C_0| \le k-1$.
For the remainder of the proof, our aim is to show that this would lead to a contradiction.

Note that $|\{1,x\} \setminus (N^c_{G_x}(1;P^*) \cup N^c_{G_x}(x;P^*) )| \ge 1$ by~\eqref{eqn:lA1B:2}.
If $ d^c_{G_x}(1)+d^c_{G_x}(x) \ge 2d - 1$, then by~\eqref{eqn:lma:k<dstronger:|S|} and~\eqref{eqn:lma:k<dstronger:|S'|}
\begin{align*}
 k-1 & \ge |C_0| =  |[1,s] \cup [w,x]| \\
	& \ge   |S| + |S^{\star}| + |\{1,x\} \setminus (S\cup S^{\star})| \\
	& \ge   2d-2k+3 + |\{1,x\} \setminus (N^c_{G_x}(1;P^*) \cup N^c_{G_x}(x;P^*) )|
	\ge 2d-2k+4,\\
	(3k-5)/2 & \ge  d, 
\end{align*}
which is a contradiction as $d$ is assumed to be at least $\lceil 3k/2 \rceil -2$.
Therefore, we may assume that $ d^c_{G_x}(1)+d^c_{G_x}(x) \le 2d - 2$.
Recall that $\delta^c(G_x) \ge d-1$. 
This means we have $d^c_{G_x}(1) = d-1 = d^c_{G_x}(x)$.
Therefore, $x+1 \in N^c_{G}(1;P^*)$ and $x+1 \in N^c_{G}(x;P^*)$ for all choices of colour neighbourhoods.
In summary, $P^*$ satisfies the following properties:
\begin{itemize}
	\item[($\alpha$)] $P^*= (1,2, \dots, x)$ is not extensible in $G_x$ and has a crossing with respect to some colour neighbourhoods;
	\item[($\beta$)] $d^c_{G_x}(1) = d-1 = d^c_{G_x}(x)$. In particular, $x+1 \in N^c_{G}(1;P^*)$ and $x+1 \in N^c_{G}(x;P^*)$ for all choices of colour neighbourhoods;
	\item[($\gamma$)] $(1,2, \dots, x, x+1, x+2,\dots,  l) \in \mathcal{R}'(P)$;
	\item[($\delta$)] $(x,x-1, \dots, 1, x+1, x+2,\dots,  l) \in \mathcal{R}'(P)$.
\end{itemize}
(Hence, our assumption that $P^* = (1,2, \dots, x)$ is in fact valid.)
Conditions ($\gamma$) and $(\delta)$ are implied by ($\beta$) and Claim~\ref{clm:lma:k<dstronger:2}~(d) and~(e).
Note that our argument after the proof of Claim~\ref{clm:lma:k<dstronger:2} actually shows that any $P' \in \mathcal{R}'(P_x)$ satisfying ($\alpha$) also satisfies ($\beta$).
Note that if $P' \in \mathcal{R}'(P_x)$ satisfies ($\alpha$)--($\delta$), then its reflection also satisfies ($\alpha$)--($\delta$).

We now mimic the proof of Lemma~\ref{lma:crossing} on $P^*$.
From now on, we further assume that $|S|= |S(P^*)| \ge |S(P'')|$ for all $P'' \in \mathcal{R}(P^*)$ satisfying ($\alpha$)--($\delta$).

If $|S| \ge 2$, then $s \ge 2$.
If $|S| =1$, then, by \eqref{eqn:lA1B:2} and taking the reflection of $P^*$ if necessary, we may assume that $s \ge 2$ .
Thus, $u$ and $w$ exist and $C_0=(1,2\dots,s,x, x-1,\dots, w, 1)$ is a p.c. cycle.
By Lemma~\ref{lma:keyproperties}~(a), the path $P^{\prime} = g_{r}(P^*) = (1, 2,\dots, r, x, x-1\dots, r+1)\in \mathcal{R}(P^*)$.
By Claim~\ref{clm:lma:k<dstronger:1}~(b), $P'$ satisfies ($\delta$) as
\begin{align}
(r+1,r+2, \dots, x, r,r-1, \dots, 1, x+1,x+2, \dots, l) \in \mathcal{R}'(P) \nonumber
\end{align}
So $N^c_{G}(r+1;P') \subseteq [x+1]$ for all choices of colour neighbourhoods by~\eqref{eqn:lma:k<dstronger:x}.
Hence, $N^c_{G_x}(r+1;P') \subseteq [x]$ for all choices of colour neighbourhoods.
Pick $N^c_{G_x}(r;P')$ with $r-1,r+1 \in N^c_{G_x}(r;P')$.
By considering 
\begin{align}
N^c_{G_x}(1;P') =
\begin{cases}
N^c_{G_x}(1;P^*) & \text{if }r \ne 1,\\
(N^c_{G_x}(1;P^*) \setminus \{2\}) \cup \{x\} & \text{if }r = 1,
\end{cases}
\label{eqn:NC}
\end{align}
$P'$ has a crossing in $G_x$.
We further deduce that $P'$ is not extensible in $G_x$ by considering~\eqref{eqn:NC} for arbitrary colour neighbourhoods.
Thus, $P'$ satisfies ($\alpha$).
This means that $P'$ also satisfies ($\beta$) and so $c(r+1,r+2) \ne c(r+1,x)$.
If $c(r+1,x) \ne c(x+1,1)$, then $(r+1,r+2, \dots, x, r,r-1, \dots, 1, x+1,r+1)$ is a p.c. cycle of length $x+1 \ge k $ by~\eqref{eqn:lma:k<dstronger:x}.
Hence, $c(r+1,x) = c(x+1,1) \ne c(x+1,x+2)$ implying that $P'$ satisfies ($\gamma$).
Let $P''$ be the reflection of $P'$, so $P'' \in \mathcal{R}(P^*)$ satisfying ($\alpha$)--($\delta$) with $N^c(i; P'') = N^c(i;P')$ for $i \in \{1, r+1\}$.

If $a \in ( N^c(1;P^*) \cap [r+1,u] ) \setminus \{2\} $, then $c(1,a) = c(a,a+1)$ by Lemma~\ref{lma:keyproperties}~(c) and~(d).
Hence, $S(P'')$ contains $( N^c(1;P^*) \cap [r+1,u] ) \setminus \{2\}$.
Since $|S|$ is maximal, 
\begin{align}
|S| \ge |S(P'')| \ge | N^c(1;P^*) \cap[r+1,u]| - \delta_{1,r} \label{eqn:|S'|:2},
\end{align}
where $\delta_{1,r} =1$ if $r=1$, and $\delta_{1,r}=0$ otherwise.
Recall \eqref{eqn:lA1B:2} that $1 \notin N^c(x;P^*)$ or $x \notin N^c(1;P^*)$, so $|\{x\}\setminus N^c(1;P^*)| - \delta_{1,r} \ge 0$.
Note that 
\begin{align}
 |C_0| & =  |[1,s]\cup [w,x]|  \nonumber \\
& = |N^c(1;P^*)| + 1+|[2,s]\setminus N^c(1;P^*)| +|[w,l] \setminus N^c(1;P^*)| - |[s+1,u] \cap N^c(1;P^*)|. \label{eqn:|C|:2}
\end{align}
By adding \eqref{eqn:|S'|:2} and~\eqref{eqn:|C|:2} together, we have
\begin{align}
 |C_0| + |S| & \ge  |N^c(1;P^*)| + 1+|[2,r]\setminus N^c(1;P^*)|+|[w,x] \setminus N^c(1;P^*)|- \delta_{1,r} +|[r+1,s]| \nonumber \\
	& \ge  d-1 + 1+ |\{x\} \setminus A|- \delta_{1,r} +|S\setminus \{r\}|  \nonumber\\
	& \ge  d-1 + |S|. \nonumber 
\end{align}
This implies $|C_0| \ge d-1 \ge k$, a contradiction.
This completes the proof of Lemma~\ref{lma:k<dstronger}.
\end{proof}

Let $G$ be an edge-coloured graph with $\delta^c(G) =d$ such that no p.c. cycle has length at least~$k$.
Suppose that one can prove that, if $d = \lceil 3k/2 \rceil-3$ then every maximal p.c. path in $G$ has length at least $k2^{d-k+2}-2$.
Then, the proof of Lemma~\ref{lma:k<dstronger} would show that Conjecture~\ref{conj:k<d} is true for all $d \ge  \lceil 3k/2 \rceil-3$.

\section{The longest p.c. path} \label{sec:path}
In this section, we prove Theorem~\ref{thm:path}.
A \emph{directed graph} is a pair $H = (V(H),A(H))$, where $V(H)$ is a finite set of vertices and $A(H)$ is a set of ordered pairs of vertices.
We refer to directed edges in $H$ as \emph{arcs}.
A graph $G$ is the \emph{base graph of $H$} if $V(G) = V(H)$ and $E(G) = \{ uv :  (u,v) \in A(H) \}$.

\begin{proof}[Proof of Theorem~\ref{thm:path}]
Let $c$ be an edge-colouring of a graph~$G$ such that $\delta^c(G) = d$.
Let $P=(1, 2, \dots, l)$ be a p.c. path in $G$ of maximum length.
Note that $P$ is maximal, so $N^c(1;P), N^c(l;P) \subseteq [l]$ implying that $l \ge d+1$.
Assume that $l <  6d/5 $, or else there is nothing to prove.
Thus, $ d \ge 6$.
We may further assume that there is no p.c. cycle $C$ spanning $[l]$.
Otherwise, $C$ is a p.c. Hamiltonian cycle if $|G| = l$ or we can find a p.c. path of length $l$ by connectedness of $G$ if $l < |G|$.
Since $l < 6d/5 $, $P$ has a crossing for all choices of $N^c(1;P)$ and $N^c(l;P)$.
By Lemma~\ref{lma:crossing}, there exist a p.c. cycle $C$ and a p.c. path $Q$ such that 
\begin{enumerate}
	\item[(i)] $C = (i_1,i_2 \dots, i_p,i_1)$ with $p \ge d+1$;
	\item[(ii)] $Q = (i'_1,i'_2 \dots, i'_q)$;
	\item[(iii)] $V(C) \cap V(Q) = \emptyset$ and $V(P) = V(C) \cup V(Q)$;
	\item[(iv)] if $q \ge 2$, then there exists $j \in [p]$ with $(i'_1,i_j) \in E(G)$ and $c(i'_1,i'_2) \ne c(i'_1,i_j)$.
\end{enumerate}
Note that $p+q =l$ and $q\ge 1$.
We may assume that $q$ is minimal.
Furthermore, by a cyclic shift and a reflection on $(i_1, i_2,\dots, i_p)$ if necessary, we may assume that $(i'_1,i_p) \in E(G)$ and if $q \ge 2$, then  
\begin{align}
c(i'_1,i'_2) \ne c(i'_1,i_p) \ne c(i_{p-1},i_p). \label{eqn:iv}
\end{align}
Therefore, $P' = (i_1,i_2 \dots, i_p,i'_1,i'_2 \dots, i'_q)$ is a p.c. path. 
Now fix $N^c(i'_q; P')$.
By the maximality of $l$, $P'$ is a maximal path and so $N^c(i'_q; P') \subseteq V(P') = V(P)$.
Note by (i) that $q = l -p \le l-d-1$.
Hence,
\begin{align*}
	|V(C) \cap N^c(i'_q; P')| \ge d - q+1 \ge 2d -l +2.
\end{align*}
Define 
\begin{align*}
	R = &\{ i_j \in V(C) \cap N^c(i'_q; P') : c(i'_q,i_j) \ne c(i_j,i_{j+1} ) \}, \text{ and} \\
	R' =& \{ i_{j -1} : i_j \in R \}, 
\end{align*}
where we take $i_{0}$ to be $i_p$.
By reversing the order of $(i_1, i_2,\dots, i_p)$ if necessary, we may assume that
\begin{align}
	|R'| = |R| \ge |V(C) \cap N^c(i'_q; P')|/2 \ge (2d-l+2)/2. \label{eqn:|R'|}
\end{align}
For distinct $i_j, i_{j'} \in R'$, observe that 
\begin{align*}
c(i_{j+1},i_{j+2}) \ne c(i'_q,i_{j+1}) \ne c(i'_q,i_{j'+1}) \ne c(i_{j'+1},i_{j'+2}).
\end{align*}
If $(i_j, i_{j'})$ is an edge with $c(i_j, i_{j-1}) \ne c(i_j, i_{j'}) \ne c(i_{j'}, i_{j'-1})$, then $i_j$ and $i_{j'}$ are at least distance 2 apart in $C$ and moreover $$C' = (i'_q, i_{j+1}, i_{j+2}, \dots, i_{j'}, i_{j}, i_{j-1}, \dots, i_{j'+1},i'_q)$$
is a p.c. cycle, see Figure~\ref{fig:thm6}.
\begin{figure}[tbp]
\begin{center}
\includegraphics[scale=0.6]{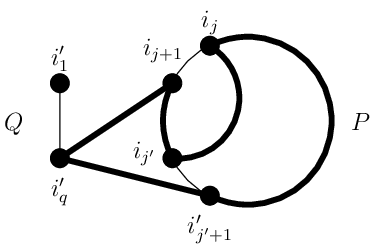}
\end{center}
\caption{Cycle $(i'_q, i_{j+1}, i_{j+2}, \dots, i_{j'}, i_{j}, i_{j-1}, \dots, i_{j'+1},i'_q)$}
\label{fig:thm6}
\end{figure}
However, this contradicts the minimality of~$q$ by setting $C = C'$ and $Q = (i'_1,i'_2 \dots, i'_{q-1})$, where condition (iv) is satisfied by~\eqref{eqn:iv}.
Thus, if $(i_j, i_{j'})$ is an edge for $i_j, i_{j'} \in R'$, then $c(i_j, i_{j'}) = c(i_{j}, i_{j-1})$ or $c(i_j, i_{j'}) = c(i_{j'}, i_{j'-1})$.

Define a directed graph $H$ on $R'$ such that there is an arc from $i_j$ to $i_{j'}$ unless $(i_j, i_{j'})$ is an edge and $c(i_j, i_{j'}) \ne c(i_{j'}, i_{j'-1})$.
Note that the base graph of $H$ is complete.
Thus, there exists a vertex $i_{j_0} \in R'$ with in-degree at least $(|R'|-1)/2$ in~$H$.
This means that
\begin{align}
& |\{i_j \in R' \setminus \{i_{j_0}\} : \text{$(i_j, i_{j_0}) \notin E(G)$ or $c(i_j, i_{j_0}) = c(i_{j_0}, i_{j_0-1})$}\}| \nonumber \\
& \ge   (|R'|-1)/2 \ge (2d-l)/4 \label{eqn:|R'|2}
\end{align}
by~\eqref{eqn:|R'|}.
Recall that $i_{j_0+1} \in R \subseteq N^c(i_q';P')$, so $c(i'_q, i'_{q-1}) \ne c(i'_q, i_{j_0+1}) \ne c(i_{j_0+1}, i_{j_0+2})$ if $q \ge 2$.
So $P'' = (i'_1, i'_2, \dots, i'_q, i_{j_0+1}, i_{j_0+2}, \dots, i_{j_0} )$ is a p.c. path.
Since $P$ is a p.c. path of maximum length, each vertex $k \in V(G)$ such that $(i_{j_0}, k )$ is an edge and $c(i_{j_0},k) \ne c(i_{j_0}, i_{j_0-1})$ must be in $V(P'') =V(P) = [l]$.
There are at least $d-1$ such vertices as $\delta^c(G) = d$.
Therefore, together with~\eqref{eqn:|R'|2} we have
\begin{align}
	l \ge |\{i_{j_0}\}| + d-1 + (2d-l)/4 = (6d-l)/4, \nonumber
\end{align}
a contradiction as $l <  6d/5 $.
\end{proof}

\section{$k$-edge-colourings} \label{sec:deg}
In this section, we consider edge-colouring with bounded number of colours.
A \emph{$k$-edge-colouring} of a graph $G$ uses $k$ colours, $c_1, c_2, \dots, c_k$.
Let $G^{c_i}$ be the subgraph of $G$ induced by edges of colour~$c_i$.
For a $k$-edge-coloured graph $G$, define $\delta^{\rm mon}_k(G) = \min \{ \delta(G^{c_i}) : 1 \le i \le k \}$.
Chapter~16 of~\cite{MR2472389} gives a good survey on~$\delta^{\rm mon}_k(G)$.

Abouelaoualim, Das, Fernandez de~la Vega, Karpinski, Manoussakis, Martinhon and Saad~\cite{Manoussakis10} proved that if $G$ is a 2-edge-coloured graph with $\delta^{\rm mon}_2(G) = \delta \ge 1$, then $G$ has a p.c. path of length $2 \delta$.
They then conjectured that if $G$ is a $k$-edge-coloured graph with $k \ge 3$ and $\delta^{\rm mon}_k(G) = \delta \ge 1$, then $G$ has a p.c. path of length $\min \{|G|-1, 2k \delta\}$.
By modifying the construction of~$\widehat{G}(n,d)$ in Example~\ref{exm:longpath}, we show that this conjecture is false.
Moreover, for every $\delta \ge 1$ and $k \ge 3$, there exists a $k$-edge-coloured graph $G$ such that $\delta^{\rm mon}_k(G) \ge \delta$ and no p.c. path has length more than $\lfloor 3\delta (k+ \eps)/2 \rfloor$, where $\eps =1$ if $k$ is even, and $\eps =0$ otherwise.
It is well known (and easy to see) that for odd $n \ge 3$ a complete graph $K_n$ of order $n$ has edge-chromatic number~$n$. 
Moreover, $K_n$ has a nearly $1$-factorization, that is an edge-decomposition into $n$ matchings each of size $(n-1)/2$, see~\cite{MR785649}.
Hence, we get the following simple fact, which we include the proof for completeness.

\begin{prp} \label{prp:K_n}
Let $n \ge 3$ be an odd integer.
Then, there exists a proper $n$-edge-colouring on $K_n$ with colours $c_1, \dots, c_k$.
Moreover, each vertex in $K_n$ misses a distinct colour.
\end{prp}

\begin{proof}
By Vizing's Theorem (see~\cite{BollobasModernGraphTheory}), the edge-chromatic number $\chi(K_n)$ is either $n-1$ or $n$.
Let $c$ be a proper edge-colouring on $K_n$ with colours $c_1, \dots, c_{\chi(K_n)}$.
Since $K_n^{c_i}$ is a matching for each $i$, $K_n^{c_i}$ spans at most $(n-1)/2$ edges.
By summing $e(K_n^{c_i})$, we deduce that $\chi(K_n) = n$.
Moreover, every $K_n^{c_i}$ is a matching of size exactly $(n-1)/2$ and the proposition follows.
\end{proof}

\begin{prp} \label{prp:counterexample}
Let $k \ge 3$ and $\delta \ge 1$ be integers.
Let $\eps =1$ if $k$ is even, and $\eps =0$ otherwise. 
Then, for each integer $n > (k+ \eps)\delta$, there exists a $k$-edge-coloured graph $G$ of order $n$ with $\delta^{\rm mon}_k(G) = \delta$ such that no p.c. path in~$G$ has length more than $\lfloor 3\delta (k+ \eps)/2 \rfloor$.
\end{prp}

\begin{proof}
Suppose $\delta=1$.
Partition $V(G)$ into $X$ and $Y$ with $X = \{x_1, x_2, \dots, x_{k+ \eps}\}$, so $Y \ne \emptyset$.
First we construct $G[X]$.
We split into two cases depending on the parity of $k$.

For $k$ odd, let $G[X]$ be a complete graph.
Since $|X| = k $ is odd, by Proposition~\ref{prp:K_n} we can properly edge-colour $G[X]$ with colours $c_1, c_2, \dots, c_k$.
Moreover, we may assume that $x_i$ is not incident with an edge of colour~$c_i$ for each $i \in [k]$.

For $k$ even, let $H$ be a complete graph on $X$.
Since $|X| = k+1$ is odd, by Proposition~\ref{prp:K_n} we can properly edge-colour $H$ with colours $ c_1, c_2, \dots, c_{k+1}$.
Again, we may assume that $x_i$ is not incident with an edge of colour~$c_i$ for each $i \in [k+1]$.
Let $G[X]$ be the $k$-edge-coloured subgraph obtained from $H$ after removing the edges with colour $c_{k+1}$, i.e. $G[X] = H - H^{c_{k+1}}$.

In summary, $G[X]$ is a properly $k$-edge-coloured graph with colours $\{ c_1, c_2, \dots, c_k\}$.
Moreover, each $x_i \in X$ is incident with edges of colours~$\{c_1, \dots, c_k\} \setminus \{c_i\}$.
Let $G[Y]$ be empty.
For each $y \in Y$, add an edge of colour $c_i$ between $y$ and $x_i$ for every $i \in [k]$.
By our construction, $\delta^{\rm mon}_k(G) = 1$.
Moreover, for each pair $y,y' \in Y$, every p.c. path from $y$ to $y'$ must contain at least two vertices in~$X$.
Thus, no path in $G$ has length more than $\lfloor 3|X|/2 \rfloor = \lfloor 3(k+ \eps)/2 \rfloor$.

For $\delta \ge 2$, the proposition is proved by considering a $\delta$-blow-up of~$G$, that is, replace  each vertex of $G$ by $\delta$ independent vertices, and add an edge of colour $c_i$ between each copy of $v$ and each copy of $u$ if and only if $u$ and $v$ are joined by an edge of colour~$c_i$ in~$G$. 
\end{proof}

On the other hand, we show that if $G$ is a $k$-edge-coloured connected graph with $\delta_k^{mon}(G) \ge \delta$, then there exists a p.c. path of length at least $(10(k-1)\delta-2)/9$ or $G$ contains a p.c. Hamiltonian cycle.

\begin{thm} \label{thm:pathmondeg}
Given an integer $k \ge 2$, every $k$-edge-coloured connected graph $G$ with $\delta^{\rm mon}_k(G) \ge 1$ contains a properly coloured path of length at least $(10(k-1)\delta^{\rm mon}_k(G) -2)/9$ or a properly coloured Hamiltonian cycle.
\end{thm}

\begin{proof}
Let $G$ be a connected graph with a $k$-edge-colouring $c$ such that $\delta^{\rm mon}_k(G) = \delta \ge 1$.
Let $P = (1, 2, \dots, l)$ be a p.c. path in $G$ of maximum length.
We may assume that $l < (10(k-1)\delta+7)/9$ or else there is nothing to prove.
We may further assume that no p.c. cycle spans the vertex set $[l]$, otherwise there exists a p.c. path of length~$l$ as $G$ is connected if $|G| > l$, or $G$ has a p.c. Hamiltonian cycle if $|G|=l$.

Define $A$ to be the set of vertices $v$ such that $(1,v)$ is an edge with $c(1,v) \ne c(1,2)$.
Note that $|A| \ge (k-1)\delta$ and $A \subseteq [3,l]$ by the maximality of $P$.
Similarly, define $B$ to be the set of vertices $v$ such that $(l,v)$ is an edge with $c(l,v) \ne c(l,l-1)$.
By a similar argument, $|B| \ge (k-1)\delta$ and $B \subseteq [l-2]$.
Let 
$$ I = \{i \in [l-1] : i+1 \in A \text{ and }i \in B \}. $$ 
Therefore,
\begin{align}
|I| \ge 2(k-1)\delta-l+1 \text{ and } I \subseteq [2,l-2]. \label{eqn:|I|}
\end{align}
Let $$I_1 = \{ i \in I \cap [3,l-3]: c(1,i+1) = c(i+1,i+2)\} \subseteq [3,l-3].$$
If there exists a vertex $i \in I$ with $c(1,i+1) \ne c(i+1,i+2)$ and $c(l,i) \ne c(i,i-1)$, then $(1,2,\dots,i,l,l-1,\dots,i+1,1)$ is a p.c. cycle spanning~$[l]$ as $i+1 \in A$ and $i \in B$.
Thus, $c(1,i+1)=c(i+1,i+2)$ or $c(l,i) = c(i,i-1)$ for all $i \in I $.
Hence, if $i \in I \setminus (I_1 \cup \{2, l-2\})$, then $c(l,i) = c(i,i-1)$.
Without loss of generality (by replacing $P$ with its reflection if necessary), we may assume by~\eqref{eqn:|I|} that
\begin{align}
|I_1| \ge (|I|-2)/2 \ge (2(k-1)\delta-l-1)/2. \nonumber
\end{align}
Since $I_1 \subseteq [3, l-3] $, there exists $i_0 \in [l]$ such $|I_1 \cap [3, i_0]|, |I_1 \cap [i_0,l-3] \ge |I_1| / 2$.
By removing at most one vertex from $I_1$, we further assume that $|I_1 \cap [3, i_0]| = |I_1 \cap [i_0,l-3] \ge (2(k-1)\delta-l-1)/4$.
Let 
\begin{align*}
R = & \{ i+2 : i \in I_1\cap [i_0,l-3] \}, \\
S = & \{ i+1: i \in I_1 \cap [3,i_0]\text{ and } c(i,l) \ne  c(i,i-1)\}, \text{ and}\\
T = &\{ i-1: i \in I_1\cap [3,i_0] \text{ and } c(i,l) =  c(i,i-1)\}.
\end{align*}
In summary, we have the following:
\begin{enumerate}
\item[(a)] for $r \in R$, $c(1,2) \ne c(1,r-1) = c(r-1,r) \ne c(r-1,r-2)$,
\item[(b)] for $s \in S$, $c(l,l-1) \ne c(l,s-1) \ne c(s-1,s-2)$,
\item[(c)] for $t \in T$, $c(l,l-1) \ne c(l,t+1) = c(t+1,t) \ne c(t+1,t+2)$,
\item[(d)] $1 <   \min\{ u \in S \cup T \} \le \max\{u \in S \cup T\} < \min\{r \in R\} \le \max \{ r \in R\} < l$,
\item[(e)] $|R| =  |S|+|T| \ge (2(k-1)\delta-l-1)/4$
\end{enumerate}
From the definition of $I$ and (a)--(c), we deduce that 
\begin{enumerate}
\item[(f)] $(r,r+1,\dots, l, s-1,s-2, \dots, 1, r-1,r-2, \dots, s)$ is a p.c. path for any $r \in R$ and $s \in S$
\item[$(g)$] $(r,r+1,\dots, l, t+1,t+2, \dots, r-1, 1, 2, \dots, t)$ is a p.c. path for any $r \in R$ and $t \in T$.
\end{enumerate}
See Figure~\ref{fig:C4}~(a) and~(b).

\begin{figure}[tbp]
\begin{center}
  \subfloat[Path $(r,r+1,\dots, l, s-1,s-2, \dots, 1, r-1,r-2, \dots, s)$]{\includegraphics[scale=0.8]{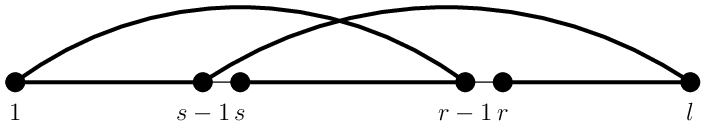}}\\
  \subfloat[Path $(r,r+1,\dots, l, t+1,t+2, \dots, r-1, 1, 2, \dots, t)$]{\includegraphics[scale=0.8]{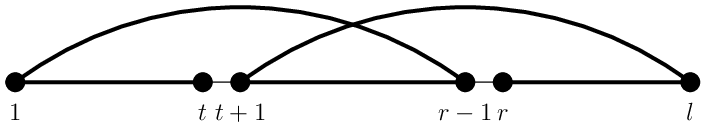}}
\end{center}
\caption{}
\label{fig:C4}
\end{figure}

Define a directed bipartite graph $H$ on vertex classes $R$ and $S \cup T$ such that for $r \in R$, $s \in S$, $t \in T$ and $u \in S \cup T$,
\begin{description}
\item[(i)] there is an arc from $u$ to $r$ unless $(r,u) \in E(G)$ with $c(r,u) \ne c(r,r+1)$,
\item[(ii)] there is an arc from $r$ to $s$ unless $(r,s) \in E(G)$ with $c(r,s) \ne c(s,s+1)$,
\item[(iii)] there is an arc from $r$ to $t$ unless $(r,t) \in E(G)$ with $c(r,t) \ne c(t,t-1)$.
\end{description}
Note that if $(r,u)$ is not an edge in $G$ for $r \in R$ and $u \in S \cup T$, then both arcs $( r , u)$ and $( u , r)$ are in $H$.
Recall that no p.c. cycle spans $[l]$.
If $(r,s)$ is an edge in $G$ for $r \in R$ and $s \in S$, then by (f) we have $c(r,s) = c(r,r+1) $ or $c(r,s) = c(s,s+1)$.
Thus, for any $r \in R $ and $s \in S$, we have $(r,s) \in A(H)$ or $(s,r) \in A(H)$.
If $(r,t)$ is an edge in $G$ for $r \in R$ and $t \in T$, then by (g) we have $c(r,t) = c(r,r+1) $ or $c(r,t) = c(t,t-1)$.
Thus, for any $r \in R $ and $t \in T$, $(r,t) \in A(H)$ or $(t,r) \in A(H)$.
Therefore, the base graph of $H$ is complete bipartite.
Moreover, if $(r,u)$ is an edge in $G$ with $r \in R$ and $u \in S \cap T$, then $c(r,u) = c(r,r+1)$ and so there is an arc from $u \in S \cap T$ to $r \in R$ in $H$.

Suppose that $H$ has maximum in-degree $\Delta_-(H)$.
Let $H' = H \setminus (S \cap T)$ and let $m = |S \cap T|$.
Given a vertex $v \in V(H')$, denote by $d_-(v)$ the in-degree of $v$ in~$H'$.
By the observation above, $d_-(r) \le (\Delta_-(H) -m)$ for all $r \in R$ and $d_-(u) \le \Delta_-(H)$ for all $u \in (S \cup T) \setminus (S \cap T)$.
Since the base graph of $H'$ is complete bipartite, by summing the in-degrees of $H'$ and (e) we have 
\begin{align*}
	|R| |(S \cup T) \setminus (S \cap T)|
	&\le e(H') \le \sum_{r \in R} d_-(r) + \sum_{u \in (S \cup T) \setminus( S \cap T)} d_-(u), \\
	 |R| (|S|+|T|-2m) & \le  (\Delta_-(H) -m)|R| + \Delta_-(H)(|S| + |T| - 2m),\\
	|R| (|R|-2m) &\le(\Delta_-(H) -m)|R| + \Delta_-(H)(|R| - 2m), \\
	 \Delta_-(H) & \ge |R|/2 \ge (2(k-1)\delta-l-1)/8,
\end{align*}
so there exists a vertex $x \in R \cup S \cup T$ with in-degree at least $(2(k-1)\delta-l-1)/8$ in~$H$.
Suppose $x \in R$.
This means that
\begin{align}
 |\{u \in S \cup T : \text{$(x, u) \notin E(G)$ or $c(x, u) = c(x, x+1)$}\}| 
& \ge  (2(k-1) \delta-l-1)/8 \label{eqn:thm:mono}.
\end{align}
Note that (f) implies that there exists a p.c. path $P' = (x,x+1, \dots)$ spanning $V(P)$.
Since $P$ is a p.c. path of maximal length, each vertex $j \in V(G)$ such that $(x, j )$ is an edge and $c(x, j) \ne c(x,x+1)$ must be in~$V(P') = [l]$.
There are at least $(k-1)\delta$ such vertices as $\delta^{\rm mon}_k(G) = \delta$.
Therefore, together with~\eqref{eqn:thm:mono} we have
\begin{align}
	l \ge |\{x\}| + (k-1)\delta + (2(k-1)\delta-l-1)/8 = (10(k-1)\delta-l +7)/8, \nonumber
\end{align}
so $l \ge (10(k-1)\delta+7)/9$, a contradiction.
A similar argument holds by (f) if $x \in S$ and by (g) if $x \in T$.
\end{proof}

\section*{Acknowledgment}

The author would like to thank S. Fujita, T.S. Tan and especially the anonymous referees for the helpful comments, the careful reviews and the suggestions on the presentation.

\end{document}